\documentclass[letterpaper, 11 pt, conference]{ieeeconf}  
                                                          
\usepackage[utf8]{inputenc}

\usepackage{graphicx}
\usepackage{amsfonts}	
\usepackage{amssymb}
\usepackage{mathtools}	
\usepackage{mathabx}
\usepackage{midfloat}
 \usepackage{stfloats}
\usepackage{algorithm}

\usepackage{amscd}										
\usepackage{booktabs}										
\usepackage{graphics} 										
\usepackage{epsfig} 										
\usepackage{contour}										
\usepackage{bm}
\usepackage{times} 											
\usepackage{rotating}										
\usepackage{cite} 									
\usepackage{listings}										
\usepackage{xcolor}

\usepackage{url}
\usepackage{array}

\newcommand{\blue}[1]{{\textcolor{black}{#1}}}

\DeclareMathAlphabet{\mathpzc}{OT1}{pzc}{m}{it}
\DeclareMathAlphabet{\mathcal}{OMS}{cmsy}{m}{n}

\newtheorem{lemma}{Lemma}
                   						
\newtheorem{remark}{Remark}
\newtheorem{theorem}{Theorem}
\newtheorem{assumption}{Assumption}

\usepackage{cases}
\usepackage{tikz}				
\usetikzlibrary{arrows,shapes,backgrounds,calc,positioning,patterns}
\usepackage{balance}

\newcommand{\qedwhite}{\hfill \ensuremath{\Box}}

\newcommand{\vect}[1]{\boldsymbol{#1}}

\DeclareMathOperator{\Dom}{Dom}	
\DeclareMathOperator{\col}{col}

\newcommand{\setof}[1]{\left\{#1\right\}} 

\newcommand{\braces}[1]{\left(#1\right)}
\newcommand{\angles}[1]{\langle#1\rangle}
\newcommand{\norm}[1]{\|#1\|}

\usepackage{pifont}

\usepackage{nicefrac}

\usepackage[nolist]{acronym}
\usepackage{hyphenat}

\allowdisplaybreaks

\title{{On modelling and stabilizability of voltage controlled piezoelectric material\\ }} 

\author{Matthijs C. de Jong and Jacquelien M. A. Scherpen 
		\thanks{ M.C. de Jong and J.M.A. Scherpen are with Jan C. Wilems Center for Systems and Control, ENTEG, Faculty of Science and Engineering, University of Groningen, Nijenborgh 4, 9747 AG Groningen, the Netherlands (email: \{  matthijs.de.jong, j.m.a.scherpen\}@rug.nl).}
	}

\begin{document}

\maketitle

    \begin{abstract}
      
       In this paper, we present a new piezoelectric actuator and piezoelectric composite model and show the well-posedness of these systems. Furthermore, we show that the piezoelectric composite is stabilizable for certain system parameters.  In this work, we also review several piezoelectric beams, actuators, and composite models and provide improved definitions of the different electromagnetic considerations, i.e. fully dynamic electromagnetic field, quasi-static electric field, and the static electric field assumption.
        
    \end{abstract}


\section{Introduction}

\begin{acronym}
    \acro{1D}{one\hyp{}dimensional}
    \acro{2D}{two\hyp{}dimensional}
    \acro{A.S.}{Assymptotically Stabilizable}
    \acro{BIBO}{Bounded Input Bounded Output}
    \acro{CbI}{Control by interconnection}
    \acro{EB}{Euler\hyp{}Bernoulli}
    \acro{EBBT}{Euler\hyp{}Bernoulli beam theory}
    \acro{IBP}{integration by parts}
    \acro{IDA}{Interconnection and Damping Assignment}
    \acro{PBC}{Passivity\hyp{}based control}
    \acro{PD}{Proportional-Derivative}
    \acro{PDE}{Partial Differential Equation}
    \acro{PDEs}{Partial Differential Equations}
    \acro{pH}{port-Hamiltonian}
    \acro{PI}{Proportional-Integral}
    \acro{PID}{Pro-portional\hyp{}Integral\hyp{}Derivative}
    \acro{PZT}{lead zirconate titanate}
    \acro{TBT}{Timoshenko beam theory}
\end{acronym}

 A piezoelectric actuator is a piece of piezoelectric material sandwiched between two layers of electrodes. An electric stimulus, such as voltage, charge, or current, can compress or elongate the actuator in one or more directions. A specific type of piezoelectric actuator is the piezoelectric beam, where an electric stimulus acting on the transverse axis incurs deformation in the longitudinal direction. By bonding a piezoelectric actuator onto the surface of a mechanical substrate, the deformation of the actuator incurs shear stress in the substrate, which curves the composition. We refer to a mechanical substrate with one or more piezoelectric actuators as a piezoelectric composite, see Fig \ref{ch2:fig:piezoelectriccomposite} and is useful in high precision applications. Piezoelectric beams and composites are considered to be capacitor dielectrics, see for instance \cite{preumont_piezoelctric2006a}.\\
 
        	\begin{figure}
    		\centering
    		\includegraphics[width=\columnwidth]{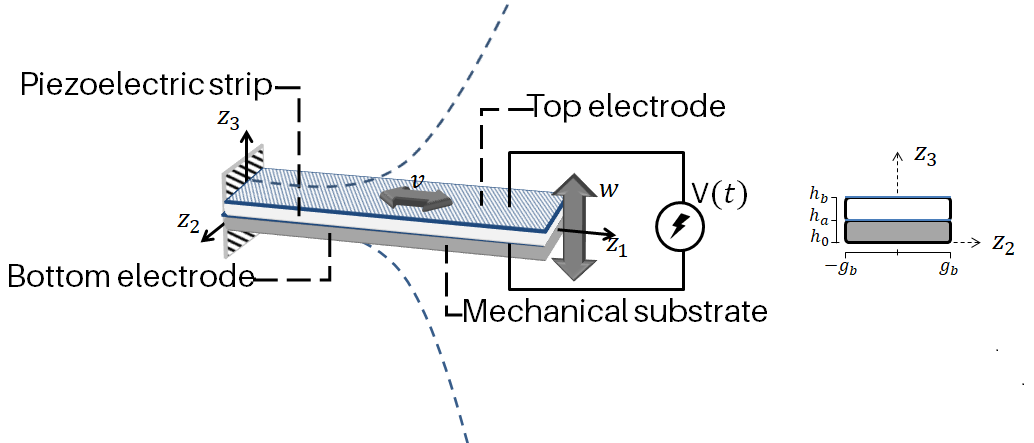}
    		\caption[Voltage-actuated piezoelectric composite]{The piezoelectric actuator consists of a piezoelectric strip, including the top and bottom electrodes. Bonding the piezoelectric actuator to a mechanical substrate results in a piezoelectric composite, which can deflect longitudinally $v$ and transversely $w$ due to the shear stresses between the two layers.}
    		\label{ch2:fig:piezoelectriccomposite}
    	\end{figure}
                      
    	From a control perspective, roughly two types of applications exist for piezoelectric actuators and composites, i.e. vibration control and shape control. The former has applications in acoustic devices as transducer \cite{Ralib2015} or suppression of vibrations in mechanical applications \cite{Samikkannu2002vibrationdamp}. The latter type envelopes applications, such as flexible wings \cite{CHUNG2009136}, inflatable space structures \cite{voss2010port}, and deformable mirrors \cite{Radzewicz04}, where high precision positioning is of great importance. Often in applications, one side of the composite has a specific function (e.g. reflecting and focusing electromagnetic waves), such as inflatable space structures and deformable mirrors.

		The dynamics describing the behaviour of a piezoelectric composite originate from continuum mechanics (mechanical domain) and Maxwell's equations (electromagnetic domain) and result in a set of \ac{PDE}. \ac{PDE} models with different structures and properties are derived by changing the mechanical or electromagnetic domain assumptions. Several different types of beam theory can be assumed for the mechanical domain. Moreover, non-linear phenomena can be incorporated. For instance, it is well-known that piezoelectric material shows complex non-linear behaviour \cite{Yang_piezoanalysis2006}, such as hysteretic behaviour caused by saturation of the polarization through high electric field strengths applied to the piezoelectric material. Moreover, drifting behaviour in strain measurements has been observed, which is known as creep. In \cite{Kugi2006}, a method which compensates for the ferroelectric hysteretic behaviour by calculating an inverse operator based on the Prandtl-Ishlinski theory is proposed.  Another approach, based on passivity, can be found in \cite{Gorbet1997Preisachstab,Gorbet2001PassivitybasedPreisach,Gorbet2003closedloopPreisach} and \cite{Valadkhan2007SIAMpassivityPreisach} for Preisach hysteresis models. In \cite{Kuhnen2007},  the inverse operator based on the Prandtl-Ishlinski theory is adjusted to compensate for the non-linear creep phenomena. Using such a systematically calculated inverse operator mitigates the non-linear behaviour, and the linear behaviour for piezoelectric beams is recovered, paving the way for stabilizability and other controller designs.\\
  
		Different assumptions on the treatment of Maxwell's equations for the electromagnetic domain lead to different dynamics and coupling of the electric, magnetic, and mechanical quantities. These assumptions are often categorised as a \textit{static electric field}, a \textit{quasi-static electric field}, or a \textit{fully dynamic electromagnetic field}. Recent efforts show that the treatment of the electromagnetic domain and choice of input severely alters the controllability and stabilizability properties of piezoelectric beams and composites \cite{MenOSIAM2014,VenSSIAM2014,Ozkan2019comp}\\
  
		The traditional choice of actuation is voltage actuation. Voltage-actuated linear infinite-dimensional piezoelectric beam models are exactly controllable and exponentially stabilizable in the case of a static or quasi-static electric field, see \cite{komornik2005}, \cite{KapitonovMM07}, and \cite{LASIECKA2009167}. In \cite{MenOSIAM2014}, the fully dynamic voltage actuated beam model is shown to be asymptotically stabilizable for almost all system coefficients and exponentially stabilizable for a small set of system coefficients. In \cite{OzerMCSS2015}, it has been shown that the fully dynamical beam is polynomially stabilizable when certain conditions on the physical coefficients are satisfied.\\
	    
        	\subsection{Contributions}
        	In this work, we focus on (linear) voltage-controlled piezoelectric beams, actuators and composites. We derive the model of a linear, fully dynamic piezoelectric actuator model, where we employ a similar modelling approach applied to a piezoelectric beam in \cite{MenOSIAM2014}. \blue{The difference between the piezoelectric actuator model and the piezoelectric beam model in \cite{MenOSIAM2014} is the incorporation of the rotational inertia, resulting in the more general piezoelectric actuator model in our case. The derived piezoelectric actuator model is more useful for developing piezoelectric composite models as we show it can be interconnected with a mechanical substrate in a straightforward manner. Furthermore, the models presented in this work differ from the non-linear (current-controlled) piezoelectric actuator and composite from \cite{VenSSIAM2014}, which are written in the port-Hamiltonian formalism \cite{port-HamiltonianIntroductory14}, through different assumptions for the mechanical domain, i.e. in \cite{VenSSIAM2014} the non-linear \ac{TBT} is used. In contrast to our linear voltage-controlled systems.} \\
                    
                The setup of this paper allows us to present a comprehensive modelling framework, which allows for a comparison of several models of piezoelectric beams and actuators from the literature in both classical and \ac{pH} formalism. \blue{Therefore, we use a hybrid formalism by writing the underlying \ac{PDEs} and aligning the variables with the energy variables of the systems to allow for a comparison of the dynamics and corresponding energy functionals.} Furthermore, we pay special attention to the treatment of the electromagnetic domain. We present the following key contributions:
                \begin{enumerate} 
                    \item we derive a more comprehensive, fully dynamic electromagnetic, piezoelectric actuator model, which takes the rotational inertia into account as opposed to the piezoelectric beam in \cite{MenOSIAM2014}, and show the system is well-posed. 
                    \item we show that the linearized fully dynamic piezoelectric beam in \cite{MenOSIAM2014} is a special case of the fully dynamic piezoelectric actuator we derive in this work through centroidal coordinates. Furthermore, we show that the fully dynamic nonlinear piezoelectric actuator and composite presented in \cite{VenSSIAM2014} exploits the so-called \textit{boundary ports} of the port-Hamiltonian (pH) framework to obtain a current input. We show that the model is part of the family of voltage-controlled piezoelectric actuators, and for the linear case, the model corresponds to our developed fully dynamic electromagnetic, voltage-controlled actuator model.
                     \item we derive a well-posed voltage-actuated, fully dynamic electromagnetic, piezoelectric composite model by interconnecting the piezoelectric actuator with a purely mechanical substrate.
                    \item {we show that the fully dynamic voltage-controlled piezoelectric composite is useful and asymptotically stabilizable.}
                    \blue{ \item we amplify the different electromagnetic assumptions, i.e. the \textit{static electric field}, the \textit{quasi-static electric field}, and the \textit{fully dynamic electromagnetic field} by giving special attention to the effect of these assumptions on the dynamical equations and corresponding interpretation.}
                \end{enumerate}
    
            \subsection{Outline}
                This paper is organised as follows. In Section \ref{ch2:sec:piezoelectric actuator}, we derive the fully dynamic piezoelectric actuator model and show that the model is well-posed. In Section \ref{ch2:sec:Piezoelectricbeams}, we reduce \blue{our} piezoelectric actuator to a piezoelectric beam and compare piezoelectric beams presented in the literature. Furthermore, we present an overview of the different electromagnetic considerations and argue for a different notion for the quasi-static electric field assumption as presented in \cite{MenOSIAM2014}. In Section \ref{ch2:sec:Piezoelectricactuators}, we make a comparison of the well-posed piezoelectric actuator derived in Section \ref{ch2:sec:piezoelectric actuator} to the findings in \cite{VenSSIAM2014}. Subsequently, in Section \ref{ch2:sec:Piezoelectriccomposites}, we proceed with the proposed piezoelectric actuator model and derive the mathematical model for a fully dynamic electromagnetic, piezoelectric composite. \blue{In Section \ref{ch2:sec:Stabz}, we show that the voltage-controlled piezoelectric composites are asymptotically stabilizable through the feedback of an electric nature and provide some illustrative simulations
                to accompany the stabilizability results in Section \ref{ch2:sec:simulations}.}
                Finally, in Section \ref{ch2:sec:DF}, we give some concluding remarks and future research directions.
    

    \section{Modelling a voltage actuated piezoelectric actuator with a fully dynamic electromagnetic field}\label{ch2:sec:piezoelectric actuator}
    
        Let us consider a volume $\Omega$ with length $\ell$, width $2g_b$, and thickness $h=h_b-h_a$ in the Cartesian coordinate system $z_1,z_2,z_3$ with unit vectors $(\vect{z_1},\vect{z_2},\vect{z_3})$, as depicted in  Fig \ref{ch2:fig:piezoelectriccomposite}. The body of $\Omega$ can be defined as follows
        \begin{align*}
             \Omega:=&\{ (z_1,z_2,z_3)\ \mid  0\leq z_1\leq \ell, \\
                 &\ -g_b\leq z_2 \leq g_b,\ h_a\leq z_3 \leq h_b  \} ,
        \end{align*}
        where we assume that the length $\ell$ is significantly larger than the width and thickness of the volume. We denote the longitudinal deformation along $z_1$ by $v(z_1,t)$, the transverse deformation along $z_3$ by $w(z_1,t)$, and the rotation of the beam given by $\phi(z_1,t)$. From now on, we omit the dependency of time $t$, when it is clear from the setting. We consider the \ac{1D} linear \ac{EBBT}, see \cite{carrera2011beam} for instance, which leads to the displacement-vector
        \begin{align}\label{ch2:eq:EB_displacement}
            \vect{u}=\begin{bmatrix}u_1 \\ u_2 \\ u_3 \end{bmatrix}= \begin{bmatrix} v(z_1) - \vect{z_3}\frac{\partial}{\partial z_1}w(z_1)\\0\\w(z_1)\end{bmatrix},
        \end{align}
        The strain $\vect{\epsilon}$, in the material is a result of (local) displacements, and is denoted by the vector 
        \begin{align}\label{ch2:eq:general_strain}
            \vect{\epsilon}=\begin{bmatrix} \epsilon_{11} & \epsilon_{22} & \epsilon_{33} & \epsilon_{23} & \epsilon_{31} & \epsilon_{21} \end{bmatrix}^T
        \end{align}
        composed of the axial strains $\epsilon_{ii}$ and shear strains $\epsilon_{ij}$ for $i\neq j$ for $i,j\in 1,\dots,3$. The subscripts $i,j$ are aligned with the coordinate system $z_1,z_2,z_3$. Using the linear strain expression  
        \begin{align}\label{ch2:eq:linengineering_strain}
            \epsilon_{ij}=\frac{1}{2}\braces{\frac{\partial u_i}{\partial z_j}+\frac{\partial u_j}{\partial z_i}}.
        \end{align} 
        on the \ac{EB} displacement vector \eqref{ch2:eq:EB_displacement}, we obtain the \ac{EB} beam strain
        \begin{align}\label{ch2:eq:EB_strain}
            \epsilon_{11}=\frac{\partial v}{\partial z_1}(z_1)-\vect{z_3}\frac{\partial^2 w}{\partial z_1^2}(z_1).
        \end{align}
        for the \ac{1D} piezoelectric actuator. The stress vector $\vect{\sigma}$ is given by
        \begin{align}\label{ch2:eq:general_stress}
            \vect{\sigma}=\begin{bmatrix} \sigma_{11} & \sigma_{22} & \sigma_{33} & \sigma_{23} & \sigma_{31} & \sigma_{21} \end{bmatrix}^T
        \end{align}
        and is related to \eqref{ch2:eq:general_strain} by the stiffness matrix $C^E$ in Hooke's law $\vect{\sigma}=C^E\vect{\epsilon}$ \cite{carrera2011beam}.\\
        
        To model the fully dynamic electromagnetic, piezoelectric actuator, we consider two types of energy; the mechanical energies, composed of the kinetic co-energy $(T^\ast)$ and the potential mechanical energy $(V)$, and the electromagnetic energies, composed of the electric energy $(\mathcal{E})$ and magnetic energy $(\mathcal{M})$. Let $\rho$ denote the mass density of the material and let the vectors $\vect{D}$, $\vect{E}$, $\vect{B}$ and $\vect{H}$ denote the electric displacement, the electric field, the magnetic field, and the magnetic field intensity, respectively. Then, the aforementioned energies are given as follows,
        \begin{subequations}\label{ch2:eq:energies_all}
            \begin{align}
                  T^\ast&=\frac{1}{2}\int_\Omega \rho(\dot{\vect{u}} \cdot \dot{\vect{u}})~d\Omega, \label{ch2:eq:energy_mechkinetic}\\
                  V&=\frac{1}{2}\int_\Omega \vect{\sigma}\cdot\vect{\epsilon}~d\Omega,\label{ch2:eq:energy_mechpotential}\\
                    \mathcal{M}&=\frac{1}{2}\int_\Omega \vect{H}\cdot\vect{B}\ d\Omega,\label{ch2:eq:Energy_Magnetic} \\
                    \mathcal{E}&=\frac{1}{2}\int_\Omega  \vect{D}\cdot \vect{E}\ d\Omega, \label{ch2:eq:Energy_Electric}
            \end{align}
        \end{subequations}
        where we omit the spatial dependency from the notation (on $z_1$). To describe the behaviour of the electromagnetic field we require Maxwell's equations for dielectrics and the piezoelectric constitutive relations for permeable material \cite{eom2013maxwell,tiersten1969linear}.
        Let $\mu$ represent magnetic permeability of the material and denote the volume charge density by $\sigma_v$. 
        Then, Maxwell's equations  \cite{eom2013maxwell} are given by the four laws;
                \begin{subequations}\label{ch2:eq:Maxwells_full}
            	    \begin{align}
                		\nabla\times \vect{E}&=-\frac{\partial \vect{B}}{\partial t},  \label{ch2:eq:Faraday'sLaw} \\
                	    \nabla\cdot\vect D&=\sigma_v, \label{ch2:eq:GausssElectric} \\
                	    \nabla\times \vect{H}&=\frac{\partial \vect D}{\partial t},\label{ch2:eq:Ampereslaw} \\
                	    \nabla\cdot\vect{B}&=0,  \label{ch2:eq:GausssMagnetic}
            	    \end{align}
            	\end{subequations}
            	corresponding with Faraday's law \eqref{ch2:eq:Faraday'sLaw} for time varying magnetic fields, Gauss's law \eqref{ch2:eq:GausssElectric} for electric fields, Ampere's law \eqref{ch2:eq:Ampereslaw} that describes the generation of a magnetic field by current densities, and Gauss's law \eqref{ch2:eq:GausssMagnetic} of magnetism, respectively. Note that the (free) current density $\vect{J}_f$ is neglected in \eqref{ch2:eq:Ampereslaw} as is common for dielectrics. However, bound charges are present as they are attached to the atoms in dielectric material \cite{IEEEstandardPiezo,tiersten1969linear,eom2013maxwell}. For the linear isotropic material properties, we consider additionally the two constitutive relations
            	\begin{subequations}\label{ch2:eq:constrel_Maxwell}
            	\begin{align}
                	\vect{D}&=\varepsilon\vect{E},\label{ch2:eq:constrel_MaxwellD}\\
                	\mu\vect{H}&=\vect{B},\label{ch2:eq:constrel_MaxwellH}
            	\end{align}
            	\end{subequations}
          In this work we consider piezoelectric beams, actuators and composites that are actuated by applying an electric potential difference (voltage $V(t)$) across the $z_3$ direction of the piezoelectric layer such that $E_3\neq 0$ and $E_1=E_2=0$. Therefore, \eqref{ch2:eq:Maxwells_full} and \eqref{ch2:eq:constrel_Maxwell} are reduced to scalar equations with $E_3, D_3, H_2$ and  $B_2$ as the remaining nonzero physical quantities by straightforward calculations. In particular, for Ampere's law \eqref{ch2:eq:Ampereslaw} and \eqref{ch2:eq:constrel_MaxwellH} we obtain the scalar equation 
         \begin{align}\label{ch2:eq:ampereslaw_scalar}
            \mu\dot{D}_3&=\frac{\partial}{\partial z_1}B_2.
        \end{align}
        Let's define the charge $q(z_1)$ similarly to \cite{MenOSIAM2014}, by
        \begin{align}\label{ch2:eq:integrated_charge}
             q(z_1)&:=\int_0^{z_1}{D}_3(\xi)d\xi,
        \end{align}
        with
        \begin{align}\label{ch2:eq:equation_qz}
             \frac{\partial q}{\partial z_1}(z_1)=D_3(z_1).
        \end{align}
        Then, integrating  \eqref{ch2:eq:ampereslaw_scalar} from $0$ to $z_1$ provides the relation between the charge $q$ and magnetic field $B_2$ as follows,
        \begin{align}\label{ch2:eq:integrated_charge_rateofhange}
            \mu \dot{q} (z_1) = B_2(z_1).
        \end{align}
        The expressions \eqref{ch2:eq:equation_qz} and \eqref{ch2:eq:integrated_charge_rateofhange} are used in writing the energies \eqref{ch2:eq:energies_all} in scalar form applicable to the considered voltage actuated piezoelectric actuator. The explicit forms of the energies \eqref{ch2:eq:energies_all} of a voltage-actuated piezoelectric actuator are as follows, 
        \begin{subequations}\label{ch2:eq:energies_all_voluit}
            \begin{align}
                  T^\ast&=\frac{1}{2}\rho \int_\Omega \dot{u}_1^2+\dot{u}_3^2~ d\Omega \nonumber \\ 
                    &=\frac{1}{2}\rho\int_\Omega \dot{v}^2+\vect{z_3}^2\dot{w}^2_{z_1}-2\vect{z_3}\dot{v}\dot{w}_{z_1}+\dot{w}^2~d\Omega. \label{ch2:eq:energy_mechkinetic_scalar}\\
                V&=\frac{1}{2}\int_\Omega \sigma_{11}\epsilon_{11}~d\Omega \nonumber\\
                    &=\frac{1}{2}\int_\Omega \sigma_{11}\braces{v_{z_1}-\vect{z_3} w_{z_1 z_1}}~d\Omega,\label{ch2:eq:energy_mechpotential_scalar}\\
                     \mathcal{M}^\ast&=\frac{1}{2}\int_{\Omega}{\frac{1}{\mu}B_2^2}d\Omega \nonumber \\
            		&=\frac{1}{2}\int_{\Omega}{\mu\dot{q}^2}d\Omega,\label{ch2:eq:Energy_Magnetic_scalar} \\
                \mathcal{E}&=\frac{1}{2}\int_\Omega  q_{z_1} {E}_3\ d\Omega, \label{ch2:eq:Energy_Electric_scalar}
            \end{align}
        \end{subequations}
        where we made use of \eqref{ch2:eq:constrel_Maxwell}, \eqref{ch2:eq:equation_qz}, and \eqref{ch2:eq:integrated_charge_rateofhange}.

            \begin{remark}\label{ch2:re:analogy}
                   In \eqref{ch2:eq:energies_all_voluit}, we indicate the (kinetic) co-energies with the use of an asterisk. The co-energy is the dual of energy and is \blue{useful} for system elements that store energy. For linear systems, the energy and co-energy are \blue{equal}; however, they are expressed in different variables. For non-linear systems, the co-energy and energy can differ. 
                   Here, we use the notion of co-energy as it is useful to differentiate between the magnetic and electric energies and their duals in relation to the Lagrangian equation. Many multi-physical systems have been derived with the use of a Lagrangian equation as the difference between the (kinetic) co-energy and potential energy. \blue{The (kinetic) co-energies are  associated with a flow (i.e. the rate of change of the generalised displacements), whereas the potential energy is associated with the generalised displacement \cite{JeltsemaMDM2009,JeltsemaSchaft_transmissionline}.}\\
            \end{remark}
        
        The coupling between the mechanic and electromagnetic domains is facilitated by the linear piezoelectric constitutive relations 
            \begin{alignat}{1}\label{ch2:eq:Constitutive_relations}
            	\left[\begin{array}{c}
                    	\vect{\sigma}\\
                    	\vect{D}
                	\end{array}\right]= & \left[\begin{array}{cc}
                    	C^E & -e\\
                    	e^T & \varepsilon
                    	\end{array}\right]\left[\begin{array}{c}
                    	\vect{\epsilon}\\
                    	\vect{E}
                	\end{array}\right],
            	\end{alignat}
        where $C^E$ is the $6\times6$ stiffness matrix, $e$ denotes the $6\times 3$ piezoelectric constants matrix, and $\varepsilon$ is the $3\times 3$ diagonal permittivity matrix, see for instance \cite{tiersten1969linear,IEEEstandardPiezo,preumont_piezoelctric2006a}. For the \ac{1D} piezoelectric {actuator} the piezoelectric constitutive relations \eqref{ch2:eq:Constitutive_relations} are reduced to the scalar equations
            \begin{align}\label{ch2:eq:preliminaries_piezoelectricConstitutiveRelaitons_scalar}
                \begin{bmatrix}\sigma_{11} \\ D_3
                \end{bmatrix}=\begin{bmatrix}
                C^E_{11} & -\gamma\\ 
                \gamma & \tfrac{1}{\beta}
                \end{bmatrix}\begin{bmatrix}\epsilon_{11}\\E_3
                \end{bmatrix},
            \end{align}
        where $\gamma:=e_{31}$ and $\beta:=\varepsilon_{33}^{-1}$ denote respectively the piezoelectric constant and impermittivity.\\

        The mathematical model of the voltage-actuated piezoelectric actuator is derived by applying Hamilton's principle to a definite integral that contains the Lagrangian $\mathcal{L}$ and the impressed forces $W$ that allows the actuation of the piezoelectric actuator by means of an applied voltage. Therefore, we
        define the definite integral 
        \begin{align}\label{ch2:eq:def_integrap_heli}
            \mathcal{J}& :=\int_{t_{0}}^{t_{1}} \mathcal{L} + W ~d t.
        \end{align}
        
        Following Remark \ref{ch2:re:analogy}, and with the generalised displacements coordinate for the electromagnetic part: the charge {\eqref{ch2:eq:integrated_charge}}, we point out that the magnetic energy \eqref{ch2:eq:Energy_Magnetic_scalar} behaves as the (kinetic) co-energy and the electric energy \eqref{ch2:eq:Energy_Electric_scalar} behaves as the potential energy for the electromagnetic domain. For the voltage-controlled piezoelectric actuator, we use the Lagrangian 
        \begin{align}\label{ch2:eq:lagrangian}
            \mathcal{L} = (T^\ast + \mathcal{M}^\ast) - (V + \mathcal{E}), 
        \end{align}
        to derive the dynamical model. The Lagrangian \eqref{ch2:eq:lagrangian} is similar to the Lagrangian in \cite{MenOSIAM2014}. However, the explicit energy expressions differ since we additionally consider the rotational inertia $I_0$, which is typical for the more general piezoelectric actuator model. Incorporating the rotational inertia's $I_0$ provides an advantage for deriving the model of a piezoelectric composite, which we illustrate in Section \ref{ch2:sec:Piezoelectriccomposites}.\\
        
        From  \eqref{ch2:eq:preliminaries_piezoelectricConstitutiveRelaitons_scalar}, we see that the coupling of the mechanical and electromagnetic domain is done on the one hand through the mechanical stress $\sigma_{11}$ present in \eqref{ch2:eq:energy_mechpotential_scalar} and on the other hand through the electric field $E_3$ present in \eqref{ch2:eq:Energy_Electric_scalar}. With the use of the energies \eqref{ch2:eq:energies_all_voluit} and the Lagrangian \eqref{ch2:eq:lagrangian}, we note that the coupling between the mechanical and electromagnetic domain is present in both the mechanical potential energy $\mathcal{V}$ and the electric energy $\mathcal{E}$. \\
        
        To \blue{actuate} the piezoelectric actuator, we define the external work $W$ next. Therefore, let $\sigma_A$ denote the surface charge and $\tilde{V}$ the electric scalar potential. Then, the external force \cite{ballas2007piezoelectric} is written as follows,
        \begin{align}\label{ch2:eq:impressedforces}
          	W=-\int_\Omega\sigma_A\frac{\partial \tilde{V}}{\partial z_3}d\Omega,
        \end{align}
        and allow the actuation of the piezoelectric actuator by means of an electric potential $V$.\\
        
        To derive the piezoelectric actuator model, we write the definite integral \eqref{ch2:eq:def_integrap_heli} as follows,
          \begin{align}\label{ch2:eq:definite_integral}
              \begin{split}
                        \mathcal{J}=& \int_{t_{0}}^{t_{1}} \frac{1}{2} \int_{0}^{l} \rho\left(A \dot{v}^{2}+I \dot{w}_{z_{1}}^{2}+2 I_{0} \dot{w}_{z_{1}} \dot{v}+A \dot{w}^{2}\right)\\
                    &+\mu A \dot{q}^{2} -{C}\left(A v_{z_{1}}^{2}+I w_{z_1 z_{1}}^{2}-2 I_{0} v_{z_{1}} w_{z_{1} z_{1}}\right)\\
                    &-\gamma \beta\left(A v_{z_1}-I_{0} w_{z, z_{1}}\right) q_{z_{1}} \\
                    &-A{\beta} q_{z_{1}}^{2}+\gamma\beta\left(A v_{z_{1}}-I_{0} w_{z_1 z_{1}}\right) q_{z_{1}} \\
                    &- A q_{z_{1}} V(t) ~d z_{1}~ d t,
                \end{split}
        \end{align}
        where we made use of the Lagrangian \eqref{ch2:eq:lagrangian}, the energies \eqref{ch2:eq:energies_all_voluit}, and work \eqref{ch2:eq:impressedforces}. Furthermore, we made use of the reinforced compliance $C:=C^E_{11}+\gamma^2\beta$ and cross-sections and inertias defined as follows,
        \begin{align}\label{ch2:eq:crossseactionalA}
                	A&:=\int_{h_a}^{h_b}\int_{-g_b}^{g_b}dz_2dz_3=2g_b(h_b-h_a), \nonumber \\
                	I&:=\int_{h_a}^{h_b}\int_{-g_b}^{g_b}{z}_3^2dz_2dz_3=\frac{2}{3}g_b(h_b^3-h_a^3),\\ 
                	I_0&:=\int_{h_a}^{h_b}\int_{-g_b}^{g_b}\vect{z}_3dz_2dz_3=\left[g_b{z}_3^2\right]_{h_a}^{h_b}=g_b(h_b^2-h_a^2). \nonumber
        	\end{align}
        Application of Hamilton's principle \cite{lanczos1970variational} to \eqref{ch2:eq:definite_integral} and setting the variation of admissible displacements $\setof{v,w_{z_1},q}$ to zero, yields the  dynamical model describing the behaviour of a voltage actuated piezoelectric actuator with fully dynamic electromagnetic piezoelectric actuator as follows,
        \begin{subequations}\label{ch2:eq:PDE-piezoactuator}
            \begin{align}\label{ch2:eq:PDE-piezoactuator-pde}
                \begin{split}
                     \rho\left(A \ddot{v}-I_{0} \ddot{w}_{z_{1}}\right)&=C\left(A {v}_{z_1 z_{1}}-I_{0} w_{z_1 z_1 z_{1}}\right)-{\gamma} \beta A q_{z_1 z_{1}}\\  
                \rho\left(I \ddot{w}_{z_{1}}-I_{0} \ddot{v }\right) &=C\left(I w_{z_1 z_{1} z_{1}}-I_{0} v_{z_{1} z_{1}}\right)+\gamma \beta I_{0} q_{z_{1} z_{1}} \\
                \mu A \ddot{q} &=\beta A q_{z_1 z_{1}}-\gamma \beta\left(A v_{z_{1} z_{1}}-I_{0} w_{z_1 z_{1} z_{1}}\right)\\
                \end{split}\\
                \begin{split}\label{ch2:eq:PDE-piezoactuator-bcs}
                v(0)=w(0)=w_{z_1}(0)=q(0)=0\\
                 {C}A v_{z_1}(\ell)-C{I_0} w_{{z_1}{z_1}}(\ell)-\gamma\beta A q_{z_1}(\ell)=0\\
                	    CI w_{{z_1}{z_1}}(\ell)-C{I_0} v_{{z_1}}(\ell)+\gamma\beta  I_0  q_{z_1}(\ell)=0\\
                	    A\beta q_{z_1}(\ell)-\gamma\beta(Av_{z_1}-I_0 w_{z_1 z_1})(\ell)=-AV(t),
                \end{split}
            \end{align}
        \end{subequations}
        	on the spatial domain $z_1\in[0,\ell]$ with total energy
         \begin{align}\label{ch2:eq:Hamiltonian_PDE_fd}
                      \mathcal{H}_{1}(t)&=\frac{1}{2} \int_{0}^{\ell}[\rho\left(A \dot{v}^{2}+I \dot{w}_{z_{1}}^{2}-2 I_{0} \dot{v} \dot{w}_{z_{1}}\right)+\mu A \dot{q}^{2} \nonumber\\
                    &+ C\left(A v_{z_{1}}^{2}+I w_{z_{1} z_{1}}^{2}-2 I_{0} v_{z_{1}} w_{z_{1} z_1}\right) \\
                    &\left.+\beta A q_{z_{1}}^{2}-2 \gamma \beta\left(A v_{z_{1}}-I_{0} w_{z_{1}z_{1}}\right) q_{z_{1}}\right] d z_{1}. \nonumber
         \end{align}
        
        We refer to the model given by \eqref{ch2:eq:PDE-piezoactuator} as a piezoelectric actuator, which is comprised of \blue{coupled} stretching and bending equations, respectively line (1,3) and (2).\\
        
        Before we show how this model relates to the existing literature, we show first the well-posedness of the derived piezoelectric actuator model \eqref{ch2:eq:PDE-piezoactuator-pde} with boundary conditions \eqref{ch2:eq:PDE-piezoactuator-bcs}. Therefore, we firstly define the operator associated with the \ac{PDE} \eqref{ch2:eq:PDE-piezoactuator} and show that the operator is, in fact, a generator of a strongly continuous semigroup of contractions, with the use of the Lumer-Philips Theorem \cite{lumer1961}.
        
            \begin{theorem}{\bf (Lumer-Phillips theorem \cite{lumer1961})}\label{ch2:theorem:Lumer-Phillips}
            The closed and densely defined operator $A$ generates a strongly continuous semigroup of contractions $T(t)$ on $X$, if and only if both $A$ and its adjoint $A^*$ are dissipative, i.e.
            \begin{align}\label{ch2:disspative_operator}
                \begin{matrix}
                        \angles{A \vect{x},\vect{x}}_X&\leq &0  \quad\text{ for all } \quad x\in\Dom(A), \\
                        \angles{A^*{\vect{x}},\vect{x}}_X&\leq &0 \quad \text{ for all }  \quad x\in\Dom(A^\ast).
                \end{matrix}        \end{align}          \qedwhite
            \end{theorem}
            Next, we establish the well-posedness of the piezoelectric actuator in the sense of semigroup theory \cite{CurtainZwart1995introduction}. \\
                
            Without loss of generality, let the length of the beam be $\ell=1$, with the spatial variable $z\in \left[0,\ell\right]$, and define the space $$H^1_0(0,1):=\setof{f\in H^1(0,1)\mid f(0)=0},$$ with $H^1(0,1)$ denoting the first order Sobolev space, moreover, let $L^2(0,1)$ denote the space of square-integrable functions. Inspired by the energy functional \eqref{ch2:eq:Hamiltonian_PDE_fd}, define the linear space
        	\begin{align*}
        	\mathcal{X}=\left\{\vect{x}\in H^1_0(0,1)\times H_0^1(0,1)\times H^1_0(0,1)\right.\\
        	\left.\times L^2(0,1) \times L^2(0,1) \times L^2(0,1) \right\}
        	\end{align*}
        	and inner product
            \begin{align}
                    \begin{split}
                                &\langle x, y\rangle_{\mathcal{X}}:=\\
                                &\left\langle\left[\begin{array}{l}
                                x_{1} \\
                                x_{2} \\
                                x_{3}
                                \end{array}\right],\left[\begin{array}{ccc}
                                C A & -C I_{0} & -\gamma \beta A \\
                                -C I_{0} & C I & \gamma \beta A \\
                                -\gamma \beta A & \gamma \beta I_{0} & \beta A
                                \end{array}\right]\left[\begin{array}{l}
                                y_{1} \\
                                y_{2} \\
                                y_{3}
                                \end{array}\right]\right\rangle_{H_{0}^{1}} \\
                                &+\left\langle \left[\begin{array}{l}
                                x_{4} \\
                                x_{5} \\
                                x_{6}
                                \end{array}\right],\left[\begin{array}{ccc}
                                \rho A & -\rho I_{0} & \\
                                -\rho I_{0} & \rho I & \\
                                & & \mu A
                                \end{array}\right]\left[\begin{array}{l}
                                y_{4} \\
                                y_{5} \\
                                y_{6}
                                \end{array}\right]\right\rangle_{{L}^{2}}\\
                                &=\int_{0}^{1}\left[C A x_{1}^{\prime} y_{1}^{\prime}+C I x_{2}^{\prime} y_{2}^{\prime}-C I_{0}\left(x_{1}^{\prime} y_{2}^{\prime}+x_{2}^{\prime} y_{1}^{\prime}\right)\right. \\
                                & \quad  -\gamma \beta A\left(x_{1}^{\prime} y_{3}^{\prime}+x_{3}^{\prime} y_{1}^{\prime}\right)+\gamma \beta I_{0}\left(x_{2}^{\prime} y_{3}^{\prime} +x_{3}^{\prime} y_{2}^{\prime}\right) \\
                                &\quad + \beta A x_3'y_3'+\rho A x_{4} y_{4}+\rho I x_{5} y_{5}-p I_{0}\left(x_{4} y_{5}+x_{5} y_{4}\right) \\
                                & \quad\left.+\mu A x_{6} y_{6}\right] d z_{1},
                    \end{split}
            \end{align}
        	where the prime indicates the spatial derivative with respect to $z_1$. The inner product $\angles{.\ ,.}_\mathcal{X}$ induces the norm $\norm{\vect{x}}^2_{\mathcal{X}}=\angles{\vect{x},\vect{x}}_\mathcal{X}=2\mathcal{H}_{1}(t)$ on $\mathcal{X}$, see \eqref{ch2:eq:Hamiltonian_PDE_fd}.
        	For simplicity, denote the spatial variable $z:=z_1$, let $\partial_{z}^i:=\frac{\partial^i}{\partial z^i}$, and define $\vect{x}:=\begin{bmatrix}v & w_{z} & q & \dot{v} & \dot{w}_{z} & \dot{q}	\end{bmatrix}^T$ to be the state. 
        		Then, the operator
        		\begin{equation}\label{ch2:eq:opA_FD}
            		\begin{gathered}
                    	\mathcal{A}:\Dom(\mathcal{A})\subset \mathcal{X}\rightarrow\mathcal{X},\\
                        	\mathcal{A}=\left[
                        	\begin{array}{ccc|ccc}
                        	& & & 1 & &\\
                        	& & & & 1 &\\
                        	& & & & & 1\\ \hline
                        \frac{C}{\rho}\partial_z^2 & & -\frac{\gamma\beta}{\rho}\partial_z^2&  & &\\
                        	& \frac{C}{\rho}\partial_z^2 & & & & \\
                        	-\frac{\gamma\beta}{\mu}\partial_z^2 & \frac{\gamma \beta I_0}{\mu A}\partial_z^2 & \frac{\beta}{\mu}\partial_z^2 & & &
                        	\end{array}
                        	\right]
                	\end{gathered}
        	\end{equation}
        	with 
        	\begin{align}\label{ch2:eq:domain_A}
                    	&\Dom(\mathcal{A})=\left\{\vect{x}\in \mathcal{X}\mid \right. \nonumber\\
                     &~\left.CA x_1'(1)-C{I_0}x_2'(1)-\gamma A x_3'(1)=0, \right.\nonumber\\
                	    &~\left.CI x_2'(1))-C{I_0}x_1'(1)+\gamma I_0  x_3'(1)=0,\right.\\ 
                	    &~\left. \beta A x_3'(1)-\gamma\beta (A  x_1'(1)-I_0 x_2'(1))=-A V(t).\right\}\nonumber
        	\end{align}
            is densely defined in $\mathcal{X}$ and describes the behaviour of a voltage-actuated piezoelectric actuator. To establish the well-posedness of the operator \eqref{ch2:eq:opA_FD} in the sense of semigroup theory, we set $V(t)=0$ and make use of the following Lemma. 
            %
            	\begin{lemma}\label{ch2:lem:FD_adjoint}
        		The adjoint $\mathcal{A}^\ast$ of the operator $\mathcal{A}$, defined in \eqref{ch2:eq:opA_FD}, is skew-adjoint. More specifically,			\begin{align*}
        		\mathcal{A}^\ast=-\mathcal{A},
        		\end{align*}
        		with $\Dom(\mathcal{A})=\Dom(\mathcal{A}^\ast)$.\\
        		\begin{proof}
        		    For any $\vect{x}=\begin{bmatrix}x_1 & \dots &x_6\end{bmatrix}^T$ and $\vect{y}=\begin{bmatrix}y_1 & \dots &y_6\end{bmatrix}^T$ $\in\Dom(\mathcal{A})$ we have,
        		    \begin{align}
            		    \begin{split}
                            \langle A \vect{x}, \vect{y}\rangle_{y}&= \int_{0}^{1}\left[C A y_{1}^{\prime} x_{4}^{\prime}-C I_{0} y_{2}^{\prime} x_{4}^{\prime}-\gamma \beta A y_{3}^{\prime} x_{4}^{\prime}\right.\\
                            &+C I y_{2}^{\prime} x_{5}^{\prime}-C I_{0} y_{1}^{\prime} x_{5}^{\prime}+\gamma \beta I_{0} y_{3}^{\prime} x_{5}^{\prime}\\
                            &+\beta A y_3^{\prime} x_{6}^{\prime}-\gamma \beta A y^{\prime}_1 x_{6}^{\prime}+\gamma \beta I_{0} y_{2}^{\prime} x_{6}^{\prime} \\
                            &+C A y_{4} x_{1}^{\prime \prime}-C I_{0} y_{5} x_{1}^{\prime \prime}-\gamma \beta A y_{4} x_{3}^{\prime \prime} \\
                            &+\gamma \beta I_{0} y_{5} x_{3}^{\prime \prime}+C I y_{5} x_{2}^{\prime \prime}-C I_{0} y_{4} x_{2}^{\prime \prime}\\
                            &+\left.\beta A y_{6} x_{3}^{\prime \prime}-\gamma \beta A y_{6} x_{1}^{\prime \prime}+\gamma \beta I_{0} y_{6} x_{2}^{\prime \prime}\right] d z_{1}\\
                            =-&\int_{0}^{1}\left[\left(C A y_{4}^{\prime}-C I_{0} y_{5}^{\prime}-\gamma \beta A y_{6}^{\prime}\right) x_{1}^{\prime}\right.\\
                            &+\left(C I y_{5}^{\prime}-C I_{0} y_{4}^{\prime}+\gamma \beta I_{0} y_{6}^{\prime}\right) x_{2}^{\prime} \\
                            &+\left(\beta A y_{6}^{\prime}-\gamma \beta\left(A y_{4}^{\prime}-I_{0} y_{5}^{\prime}\right)\right) x_{3}^{\prime}\\
                            &+\left(C A y_{1}^{\prime \prime}-C I_{0} y_{2}^{\prime \prime}-\gamma \beta A y_{3}^{\prime \prime}\right) x_{4} \\
                            &+\left(C I y_{2}^{\prime \prime}-C I_{0} y_{1}^{\prime \prime}+\gamma \beta I_{0} y_{3}^{\prime \prime}\right) x_{5} \\
                            &+\left(\beta A y_{3}^{\prime \prime}-\gamma \beta\left(A y_{1}^{\prime \prime}-I_{0} y_{2}^{\prime \prime}\right) x_{6}\right] d z_{1}\\
                            &+\left[\left(C A y_{1}^{\prime}-C I_{0} y_{2}^{\prime}-\gamma \beta A y_{3}^{\prime}\right) x_{4}\right. \\
                            &+\left(C I y_{2}^{\prime}-C I_{0} y_{1}^{\prime}+\gamma \beta I_{0} y_{3}^{\prime}\right) x_{5} \\
                            &+\left(\beta A y_{3}^{\prime}-\gamma \beta A y_{1}^{\prime}+\gamma \beta I_{0} y_{2}^{\prime}\right) x_{6} \\
                            &+\left(C A x^{\prime}_1+C I_{0} x_{2}^{\prime}-\gamma \beta A x_{3}^{\prime}\right) y_{4} \\
                            &+\left(C I x_{2}^{\prime}-C I_{0} x_{1}^{\prime}+\gamma \beta I_{0} x_{3}^{\prime}\right) y_{5} \\
                            &\left.+\quad\left(\beta A x_{3}^{\prime}-\gamma \beta\left(A x_{1}^{\prime}-I_{0} x_{2}^{\prime}\right)\right) y_{6}\right]_{0}^{1}\\
                            &=\langle \vect{x},-\mathcal{A}\vect{y}\rangle_\mathcal{X},
                        \end{split}
        		    \end{align}
        		    where we made use of \ac{IBP}, the domain of $\mathcal{A}$ \eqref{ch2:eq:domain_A}, and imposed the boundary conditions
        		    \begin{align*}
        		        &y_4(0)=y_5(0)=y_6(0)=0,\\
        		        & CA y_1'(1)-C{I_0}y_2'(1)-\gamma A y_3'(1)=0, \\
        		        & CI y_2'(1))-C{I_0}y_1'(1)+\gamma I_0  y_3'(1)=0, \\
        		        &\beta A y_3'(1)-\gamma\beta (A  y_1'(1)-I_0 y_2'(1)=0,
        		    \end{align*}
        		    for $\mathcal{A}^\ast$, similar to the boundary conditions of $\mathcal{A}$. Hence, we conclude the skew-symmetry of $\mathcal{A}^\ast$, i.e. $\mathcal{A}^\ast=-\mathcal{A}$ and $\Dom(\mathcal{A}^\ast)=\Dom(\mathcal{A})$.
        		\end{proof}
        	\end{lemma}
            
            Now we are able to establish the well-posedness of the voltage-actuated piezoelectric actuator in the absence of control.
        	\begin{theorem}\label{ch2:theorem:contractionFD} 
        	    The operator $\mathcal{A}$, defined in \eqref{ch2:eq:opA_FD}, generates a semigroup of contractions,     satisfying $\norm{T(t)}\leq 	1$ on $\mathcal{X}$.
        	\end{theorem}
        	\begin{proof}
        		The closed and densely defined operator $\mathcal{A}$ satisfies
        		\begin{align}\label{ch2:eq:A_dissipative}
                		&\angles{\mathcal{A}\vect{x},\vect{x}}_\mathcal{X}=\int_{0}^{1} -\left(\left(A x_{1}^{\prime \prime}-C I_{0} x_{2}^{\prime \prime}-\gamma \beta A x_{3}^{\prime \prime}\right) x_{4}\right.\nonumber\\
                        &\quad-\left(C I x_{2}^{\prime \prime}-C I_{0} x_{1}^{\prime \prime}+\gamma \beta I_{0} x_{3}^{\prime \prime}\right) x_{5} \nonumber\\
                        &\quad-\left(\beta A x_{3}^{\prime \prime}-\gamma \beta\left(A x_{1}^{\prime \prime}-I_{0} x_{2}^{\prime \prime}\right) \right) x_{6}\nonumber\\
                        &\quad+\left(\left(A x_{1}^{\prime \prime}-C I_{0} x_{2}^{\prime \prime}-\gamma \beta A x_{3}^{\prime \prime}\right) x_{4}\right.\nonumber\\
                        &\quad+\left(C I x_{2}^{\prime \prime}-C I_{0} x_{1}^{\prime \prime}+\gamma \beta I_{0} x_{3}^{\prime \prime}\right) x_{5} \nonumber\\
                        &\quad+\left(\beta A x_{3}^{\prime \prime}-\gamma \beta\left(A x_{1}^{\prime \prime}-I_{0} x_{2}^{\prime \prime}\right) \right) x_{6}\nonumber\\
                        &\quad+[\left(C A x_{1}^{\prime}-C I_{0} x_{2}^{\prime}-\gamma \beta A x_{3}^{\prime}\right) x_{4} \nonumber\\
                        &\quad+\left(C I x_{2}^{\prime}-C I_{0} x_{1}^{\prime }+\gamma \beta I_{0} x_{3}^{\prime}\right) x_{5} \nonumber\\
                        &\quad+\left(\beta A x_{3}^{\prime}-\gamma \beta\left(A x_{1}^{\prime}-I_{0} x_{2}^{\prime}\right) x_{6}\right]_{0}^{1} \nonumber\\
                        &\quad=0\leq 0
        		\end{align}
            	where we made use of \ac{IBP} and the domain of $\mathcal{A}$. Furthermore, from Lemma \ref{ch2:lem:FD_adjoint}, we have that
        		\begin{align}\label{ch2:eq:Aast_dissipative}
        		    \langle \mathcal{A}^\ast \vect{x},\vect{x}\rangle_\mathcal{X}=\langle -\mathcal{A} \vect{x},\vect{x}\rangle_\mathcal{X}=-\langle \mathcal{A} \vect{x},\vect{x}\rangle_\mathcal{X}=0\leq 0
        		\end{align}
        		\blue{In \eqref{ch2:eq:A_dissipative} and \eqref{ch2:eq:Aast_dissipative}, we have shown that both $\mathcal{A}$ and $\mathcal{A}^\ast$ are dissipative. Hence, with use of the Lumer-Phillips Theorem (Theorem \ref{ch2:theorem:Lumer-Phillips}), we conclude that the operator \eqref{ch2:eq:opA_FD} generates a semigroup of contractions and is therefore well-posed.} 
        	\end{proof}
       
        In the sections \ref{ch2:sec:Piezoelectricbeams}, \ref{ch2:sec:Piezoelectricactuators}, \ref{ch2:sec:Piezoelectriccomposites}, we use the piezoelectric actuator and show how this model relates to the piezoelectric beams, actuators and composites models in the literature, respectively. Furthermore, we will touch upon the different electromagnetic considerations for modelling piezoelectric beams, actuators, and composites. Furthermore, in Sections \ref{ch2:sec:Stabz}, we show the stabilizability results of our model and add some simulations for simulation and control purposes in Section \ref{ch2:sec:simulations}.\\

    \section{Comparison of piezoelectric beams } \label{ch2:sec:Piezoelectricbeams}
    
        In this section, we relate the derived model by \blue{comparing} it to existing piezoelectric beam models. A piezoelectric actuator can either be interconnected with a mechanical substrate to obtain a piezoelectric composite, the piezoelectric actuator can be interconnected at the boundary with its environment for specific applications, or the piezoelectric actuator can be considered on its own by taking centroidal coordinates (i.e. $h_b=-h_a$). The latter ensures that the cross-sectional inertia $I_0\rightarrow 0$, see \eqref{ch2:eq:crossseactionalA}. Taking centroidal coordinates reduces the piezoelectric actuator \eqref{ch2:eq:PDE-piezoactuator} to the fully dynamic electromagnetic piezoelectric beam presented in \cite{MenOSIAM2014} by substituting $I_0=0$ into \eqref{ch2:eq:PDE-piezoactuator-pde} and \eqref{ch2:eq:PDE-piezoactuator-bcs}, respectively. \blue{The bending equations become decoupled from the voltage actuation and the stretching equations \eqref{ch2:eq:PDE-piezobeam} by taking centroidal coordinates, see \cite{MenOMTNS2014} for more detail. The stretching equations remain influenced by the control input $V(t)$ and are as follows,}
        \begin{subequations}\label{ch2:eq:PDE-piezobeam}
            \begin{align}\label{ch2:eq:PDE-piezobeam-pde}
                \begin{split}
                    \rho\ddot{v}&=Cv_{z_1z_1}-\gamma\beta q_{z_1z_1}\\
                    \mu\ddot{q}&=\beta q_{z_1z_1}-\gamma\beta v_{z_1z_1}
                \end{split}\\ \nonumber \\
                \begin{split}\label{ch2:eq:PDE-piezobeam-bcs}
                    v(0)&=q(0)=0\\
                    C v_{z_1}&(1)-\gamma\beta q_{z_1}(1)=0\\
                    \beta q_{z_1}&(1)-\gamma\beta v_{z_1}(1)=-V(t).
                \end{split}    
            \end{align}
        \end{subequations}
        \blue{The decoupling of the bending equations from the stretching equations and voltage-actuation also holds for other beam theories, such as the Timoshenko or Rayleigh beam theories \cite{MenOSIAM2014,MenOACC2014}. Furthermore, in \cite{MenOACC2014}, it is shown that modelling voltage-controlled piezoelectric beams using either the Euler-Bernoulli or Timoshenko beam theory results in exactly the same dynamical equations, due to the decoupling of the stretching and bending equations. The piezoelectric beam model \eqref{ch2:eq:PDE-piezobeam}, is similar to the piezoelectric beam model in \cite{MenOMTNS2014}, which is derived by neglecting the rotational inertia (similar to taking centroidal coordinates) before applying Hamilton's principle to \eqref{ch2:eq:definite_integral}, see \cite{MenOSIAM2014}. The fully dynamic piezoelectric beam \eqref{ch2:eq:PDE-piezobeam} is well-posed and exponentially stabilizable for special cases of the system parameters  \cite{MenOSIAM2014}. For another set of system parameters, the system is strongly stabilizable \cite{MenOSIAM2014,MenOCDC2013}.} In this section, we have shown the relation of the piezoelectric actuator to the piezoelectric beam. In the next section, we continue with the different electromagnetic considerations and their connection to the dynamics of piezoelectric beams and actuators.\\ 
    %
    
        \subsection{Electromagnetic considerations}\label{ch2:sec:electromagneticconsiderations}

        Three considerations exist for the electromagnetic domain: a \textbf{fully dynamic electromagnetic field}, a \textbf{quasi-static electric field}, and a \textbf{static electric field}, which we touch upon in this section. In this order, they take less and less dynamic behaviour of the electromagnetic domain into account. \\
            \subsubsection{The fully dynamic electromagnetic field}            
            For both the voltage actuated piezoelectric actuator \eqref{ch2:eq:PDE-piezoactuator} and piezoelectric beam \eqref{ch2:eq:PDE-piezobeam} a fully dynamic electromagnetic field is assumed. The electromagnetic field is produced by accelerating electric charges and by Maxwell's correction (i.e. $\tfrac{\partial \vect{D}}{\partial t}$) to Ampere's original law, ensuring the continuity equation
            \begin{align}\label{ch2:eq:continuity_eq}   \nabla\cdot\vect{J}=-\frac{\partial \sigma}{\partial t},
            \end{align}
            which amounts to the conversation of charge. 
            \begin{remark}\label{ch2:rem:bound&free_currentdensity}
               The current density is composed of free and bound current densities, i.e.
               \begin{align*}
                   \vect{J}=\vect{J}_f+\vect{J}_b
               \end{align*}
                For dielectric material, all the charges within the material are bound to the atoms. Therefore, the current densities present within piezoelectric material are bound, i.e. $\vect{J}\equiv \vect{J}_b$. It is also possible to actuate a piezoelectric beam or actuator by a current, which allows free charges to flow through the piezoelectric. In that case $\vect{J}_f\neq 0$. However, the scope of this work encompasses voltage-actuated piezoelectric beams, actuators and composites.  
            \end{remark}
            From \eqref{ch2:eq:GausssMagnetic}, it is known that there exist a magnetic vector potential $\vect{A}$ and electric scalar potential $\varphi$ such that $\vect{B}=\nabla\times \vect{A}$ \cite{eom2013maxwell,tiersten1969linear,MenOSIAM2014}. Therefore, through Ampere's law \eqref{ch2:eq:Ampereslaw} and the use of vector identities, the expression for the electric field
            \begin{align}\label{ch2:eq:E-fieldFD}
                E&=-\nabla \varphi -\frac{\partial \vect{A}}{\partial t}
            \end{align}
            is derived. In \eqref{ch2:eq:E-fieldFD}, the dynamic contribution of the magnetic field to the electric field for the case of a fully dynamic electromagnetic case is shown. In piezoelectric beams and actuators, the fully dynamic electromagnetic field is characterised through Faraday's law \eqref{ch2:eq:Faraday'sLaw} by 
            \begin{subequations}
                \begin{align}
                    \frac{\partial B_2}{\partial t}&=\frac{\partial E_3}{\partial z_1} \label{ch2:eq:characterizing_eq_1}  \\ 
                    &=\beta\frac{\partial D_3}{\partial z_1}-\gamma\beta \frac{\partial \epsilon_{11}}{\partial z_1} \label{ch2:eq:characterizing_eq_2} \\
                    \mu A\frac{\partial^2 q}{\partial t^2}&=\beta A \frac{ \partial^2 q }{\partial z_1^2 } -\gamma\beta\braces{A\frac{\partial^2 v}{\partial z_1^2}-I_0\frac{\partial^3 w}{\partial z_1^3}} \label{ch2:eq:characterizing_eq_3},
                \end{align}
            \end{subequations}
            where we made use of $E_1=E_2=0$ in \eqref{ch2:eq:Faraday'sLaw} to obtain \eqref{ch2:eq:characterizing_eq_1}. Furthermore, in \eqref{ch2:eq:characterizing_eq_2}, we made use of \eqref{ch2:eq:preliminaries_piezoelectricConstitutiveRelaitons_scalar}, and in \eqref{ch2:eq:characterizing_eq_3}, we used the relations \eqref{ch2:eq:equation_qz} and \eqref{ch2:eq:integrated_charge} on both sides and integrate  over the cross-section $A$. The characterising expression \eqref{ch2:eq:characterizing_eq_3} recurs in the dynamics of the piezoelectric actuator\eqref{ch2:eq:PDE-piezoactuator} and beam \eqref{ch2:eq:PDE-piezobeam} (with $I_0\downarrow 0$).\\
            {\subsubsection{The quasi-static electric field} Besides the fully dynamic electromagnetic field assumption, a static electric field or quasi-static electric field assumption can be used \cite{IEEEstandardPiezo,tiersten1969linear,eom2013maxwell}. The quasi-static electric field is obtained by restricting the magnetic field intensity to Ampere's original law for magnetism \cite{eom2013maxwell} through, 
            \begin{align}\label{ch2:eq:Amperes-original}
                \nabla\times \vect{H}\equiv 0,
            \end{align}
            such that the electric displacement becomes static, i.e. $\dot{\vect{D}}\equiv0$. However, this does not imply that the electric field $\vect{E}$ nor the magnetic field $\vect{B}$ becomes static. Moreover, \eqref{ch2:eq:GausssElectric} remains valid as opposed to the definition of the quasi-static electric in \cite{MenOSIAM2014}. We argue that there are implicit magnetization and polarization properties present in the piezoelectric material \cite{tiersten1969linear,IEEEstandardPiezo}. Let's denote the magnetic-field by $\vect{M}$, the electric susceptibility by $\chi$, and define the polarization $\vect{P}\equiv \varepsilon_0\chi \vect{E}$, such that
            \begin{subequations}
                \begin{align}
                    \vect{D}&=\varepsilon_0 \vect{E}+\vect{P}, \label{ch2:eq:polarization}\\
                    \vect{B}&=\mu_0\braces{\vect{H}+\vect{M}},\label{ch2:eq:magnetization}
                \end{align}
            \end{subequations}
            where $\varepsilon_0$ and $\mu_0$ denote the respective permittivity and permeability in free space. Furthermore, note that $\varepsilon=\epsilon_0(1+\chi)$. Then, by substituting \eqref{ch2:eq:polarization} and \eqref{ch2:eq:magnetization} in both sides of \eqref{ch2:eq:Ampereslaw}, we obtain 
             \begin{align}\label{ch2:eq:proofapoint_continuityeq}
                 \nabla\times\braces{\mu_0 \vect{ B}}-\underbrace{\nabla\times\vect{M}}_{\vect{J}_M}&=\braces{1+\chi}^{-1}\frac{\partial \vect{D}}{\partial t}+\underbrace{\frac{\partial \vect{P}}{\partial t}}_{\vect{J}_P},
             \end{align}
            with $\vect{J}_M$ and $\vect{J}_P$ denoting the respective magnetic and polarizing current densities, which are bound electric current densities, i.e. $\vect{J}_b=\vect{J}_M+\vect{J}_P$. Taking the divergence on both sides of \eqref{ch2:eq:proofapoint_continuityeq} recovers the continuity equation \eqref{ch2:eq:continuity_eq} as follows,
            \begin{align}\label{ch2:eq:continuity_proofapoint}
                \begin{split}
                    \nabla\cdot\vect{J}_P &= -\braces{1+\chi}^{-1}\frac{\partial }{\partial t}\braces{\nabla \cdot \vect{D}}\\
                    &=-\frac{\partial}{\partial t}\underbrace{\braces{1+\chi}^{-1}\sigma_v}_{\sigma},
                \end{split}
            \end{align}
            where we made use of \eqref{ch2:eq:GausssElectric}. This proves that the continuity equation remains valid for the quasi-static electric field assumption as opposed to the definition given in \cite{MenOSIAM2014}, as a consequence of the present bound charges due to the polarization properties of the material. Furthermore, from \eqref{ch2:eq:Amperes-original} we know there exists a  magnetic scalar potential $\Phi_M$, such that $\vect{H}\equiv -\nabla \Phi_M$, and subsequently, 
            \begin{align}
                \vect{B}&=\mu_0\braces{\vect{M}-\nabla\Phi_M},\label{ch2:magneticmagentization}
            \end{align}
            showing the influence of the magnetization on the magnetic field $\vect{B}$.  Due to the coupling of the electromagnetic and mechanical domains by \eqref{ch2:eq:Constitutive_relations}, we see that polarization and magnetization occur implicitly under the quasi-static electric field assumption.
            More precisely, in \eqref{ch2:eq:Constitutive_relations}, the direct coupling between the electric field $\vect{E}$ and the strain $\vect{\epsilon}$ for $\vect{D}\equiv 0$ is the effect of the magnetization and polarization effects. For the case of piezoelectric beams and actuators, we obtain the constraint $E_3=-\gamma\beta\epsilon_{11}$ from \eqref{ch2:eq:preliminaries_piezoelectricConstitutiveRelaitons_scalar} combined with $D_3= q_z \equiv 0$, which can be substituted in the dynamics \eqref{ch2:eq:PDE-piezoactuator-pde} and \eqref{ch2:eq:PDE-piezobeam} to obtain the dynamical model for a piezoelectric actuator and beam model under a quasi-static electric field assumption, respectively. 
            {\subsubsection{The static electric field} The last simplification of the electromagnetic domain to consider is the static electric field assumption, which considers the charges to be at rest and is obtained by restricting the electric field through $\nabla\times \vect{E}=0$. Therefore, with the use of Poincaré's theorem \cite{Kreyszig1989introductory}, the electric field expression is reduced to the static electric field relation 
            \begin{align}\label{ch2:eq:E-fieldStat}
                E=-\nabla \varphi,
            \end{align}
            see for instance \cite{tiersten1969linear,eom2013maxwell,MenOSIAM2014}. \blue{Then, the magnetic contribution to the electric field $\vect{E}$ is neglected, i.e. $\dot{\vect{A}}\equiv 0$, for the static electric field assumption, in contrast to \eqref{ch2:eq:E-fieldFD} for the fully dynamic case.} Moreover,
            from Faraday's law \eqref{ch2:eq:Faraday'sLaw} it follows that the magnetic field $\vect{B}$ becomes stationary ($\dot{\vect{B}}=0$) and without loss of generality $\vect{B}\equiv 0$. Therefore, 
            \begin{align*}
                \frac{\partial B_2}{\partial t}=\mu \ddot{q}=0,
            \end{align*}
            such that the right hand side of \eqref{ch2:eq:characterizing_eq_2} becomes an algabraic constraint. For piezoelectric beams and actuators, the constraint $\beta\tfrac{\partial D_3}{\partial z_1}=\gamma\beta\tfrac{\partial \epsilon_{11}}{\partial z_1}$ can be substituted in both \eqref{ch2:eq:PDE-piezoactuator} and \eqref{ch2:eq:PDE-piezobeam}, to obtain the respective dynamics for a piezoelectric actuator and piezoelectric beam under the static electric field assumption.}\\
                        
            The obtained models for piezoelectric beams for the cases of a static electric field and quasi-static electric field are quite similar, i.e.
            \begin{align}\label{ch2:eq:(Q)staticbeams}
                \rho \ddot{v}&=a v_{z_1z_1}\\
                v(0)&=0, \quad Cv_{z_1}=-\gamma V(t)
            \end{align}
            with $a\equiv C_{11}$ or $a\equiv C$, respectively. Under the quasi-static electric field assumption, the stiffness is slightly higher (i.e. $C=C_{11}+\gamma^2\beta$) than for the models with a static electric field assumption. The systems \eqref{ch2:eq:(Q)staticbeams} are known to be observable and exponentially stabilizable, see for instance \cite{LASIECKA2009167,KapitonovMM07,komornik2005}, making these models interesting for controller design, as is shown in \cite{Macchelli2017_BCLDPS_dissipation,KosarajudejongECC2019}. This concludes the comparison of piezoelectric beams. In the next section, we continue with the comparison of piezoelectric actuators.\\
            }

        
    \section{Comparison of piezoelectric actuators}\label{ch2:sec:Piezoelectricactuators}

        In this section, we analyse the nonlinear piezoelectric actuator from \cite{VenSSIAM2014} and relate it to the piezoelectric actuator \eqref{ch2:eq:PDE-piezoactuator-pde} and to the beam models in \cite{MenOSIAM2014}, for completeness. The nonlinear piezoelectric actuator derived in \cite{VenSSIAM2014} is a current controlled nonlinear fully dynamic piezoelectric actuator model and is derived by interconnecting the models of a transmission line \cite{JeltsemaSchaft_transmissionline} with the dynamics of a nonlinear Timoshenko beam within the \ac{pH}-framework. The resulting interconnected model describes the dynamics of a nonlinear piezoelectric actuator. With the use of linearization and reduction of the \ac{TBT} to the \ac{EBBT}, we show that the current-controlled nonlinear piezoelectric actuator derived in \cite{VenSSIAM2014} is of the same family as the derived voltage-controlled linear piezoelectric actuator \eqref{ch2:eq:PDE-piezoactuator}. 
        
        To describe the nonlinear piezoelectric actuator model from \cite{VenSSIAM2014}, which is based on the \ac{TBT} \cite{carrera2011beam}, we define $\phi$ as the transverse rotation of the beam and define $G$ as the shear-strain coefficient. Now we can describe the \ac{PDE} of the nonlinear current controlled nonlinear piezoelectric actuator and its boundary conditions present underneath the \ac{pH} nonlinear model presented in \cite{VenSSIAM2014} as follows,
        
        \begin{subequations}\label{ch2:eq:PDE-nonlinVoss}
        \small{
            \begin{align}
                \begin{split}\label{ch2:eq:PDE-nonlinVoss-pde}
                    \rho(A\ddot{v}-I_0\ddot{\phi})=\frac{\partial}{\partial z}\braces{C\braces{Av_z-I_0\phi_z+\tfrac{1}{2}Aw_z^2}-\gamma\beta A q_z}&\\
                    \rho A \ddot{w}= \frac{\partial}{\partial z}\left(\braces{C\braces{Av_z-I_0 \phi_z+\tfrac{1}{2}Aw_z^2}-\gamma\beta A q_z}w_z\right.
                     \\\left.+\tfrac{1}{2}AG\braces{w_z-\phi}\right)&\\
                     \rho(A\ddot{\phi}-I_0\ddot{v})=\frac{\partial}{\partial z}\left(C\braces{I\phi_z-I_0\braces{v_z-\tfrac{1}{2}w_z^2}}+\gamma\beta I_0 q_z\right.\\\left.-\tfrac{1}{2}AG\braces{\phi-w_z}\right)&\\
                     \mu A \ddot{q} = \frac{\partial}{\partial z}\left( \beta A q_z -\gamma\beta \braces{Av_z-I_0\phi_z+\tfrac{1}{2}Aw_z^2} \right)&
                \end{split}\\
                \begin{split}\label{ch2:eq:PDE-nonlinVoss-bc}
                     v(0)=w(0)=w_z(0)=\phi(0)=0&;\\
                     \beta A q_z(0)-\gamma\beta\left(Av_z(0)-I_0\phi_z(0) \right)=0&\\
                     C\braces{Av_z(1)-I_0\phi_z(1)+\tfrac{1}{2}Aw_z^2(1)}-\gamma\beta Aq_z(1)=0&\\
                     C\left(I\phi_z(1)-I_0\braces{v_z(1)-\tfrac{1}{2}w_z(1)^2}\right)+\gamma\beta I_0q_z(1)=0&\\
                     C\left(Aw_z^3(1)-I_0w_z(1)\phi_z(1)+Av_z(1)w_z(1)\right)\quad\quad\quad\quad &\\
                     +\tfrac{1}{2}AG\braces{w_z(1)-\phi(1)}-\gamma\beta A q_z(1){w_z(1)}=0&\\
                     \dot{q}(1)=I(t)&,
                 \end{split}
            \end{align}
            }
        \end{subequations}
        with cross-section and inertia as in \eqref{ch2:eq:crossseactionalA}. In the last line of \eqref{ch2:eq:PDE-nonlinVoss-bc}, the current input is present as boundary input through exploiting the so-called boundary ports of the \ac{pH}-formalism; see Remark \ref{ch2:rem:currentthroughtheboundary} for more details.

        The total energy corresponding to the system described by \eqref{ch2:eq:PDE-nonlinVoss} is
        \begin{align}
                 \mathcal{H}_{2}(t)&=\frac{1}{2}\int_0^1 \rho\left( A\dot{v}^2+I\dot{\phi}^2 -2I_0\dot{v}\dot{\phi}+A\dot{w}\right)+\mu A\dot{q}^2 \nonumber\\
                &+C\left(Av_z^2+I \phi_z^2 +\tfrac{1}{4}Aw_z^4-2I_0\braces{v_z-\tfrac{1}{2}w_z^2}\phi_z\right.\nonumber\\
                & \left.+Av_zw_z^2\right)+\tfrac{1}{2}AG\braces{w_z^2+\phi^2-2w_z\phi}+\beta A q_z^2\nonumber\\
                &-2\gamma\beta\braces{Av_z-I_0\phi_z+\tfrac{1}{2}Aw_z^2}q_z  ~dz.
        \end{align}

            \begin{remark}\label{ch2:rem:currentthroughtheboundary}
                Within the class of distributed parameter systems of the \ac{pH}-formalism, the inputs and outputs of a system are represented by the boundary ports, which are comprised of the flows (\textit{the rate of change of the generalized coordinates}) and efforts (\textit{forces}) on the boundary and are well established \cite{port-HamiltonianIntroductory14}. For the models \eqref{ch2:eq:PDE-nonlinVoss-bc} and \eqref{ch2:eq:PDE-piezoactuator}, the voltage input is a consequence of a force balance and is related to the effort boundary port (within the \ac{pH}-framework). The current control input of the piezoelectric model in \cite{VenSSIAM2014} is associated with the flow boundary port as the rate of change of the charge $q$. In the classical formulation, the flow boundary port is closely related to a Dirichlet boundary condition. In a simple case, we have, for instance, $q(a)=c$, where c is a constant. Therefore, $\dot{q}(a)=0$ with $a$ on the boundary. For a control input $u$ we get a Dirichlet boundary condition $q(1)=u$ with $u\equiv y$ being the output of the single integrator system with state $\nu$ and
                \begin{align}\label{ch2:eq:singleintegrator}
                \begin{split}
                     \dot{\nu}&=I(t)\\
                    y&=\nu.
                \end{split}
                \end{align}
                In summary, the current input of the \ac{pH} system in \cite{VenSSIAM2014} is represented by a single integrator system on the boundary within the classical formalism. Therefore, we refer to this type of current input as a \textit{current-through-the-boundary} input. Piezoelectric beams and composites with this type of current activation have been investigated in \cite{MenOCDC2014,OzerTAC2019}.
            \end{remark}
    
        To make the comparison between the linear well-posed piezoelectric actuator model \eqref{ch2:eq:PDE-piezoactuator-pde} and the nonlinear piezoelectric model \eqref{ch2:eq:PDE-nonlinVoss} we set both inputs to zero, i.e. $V(t)=0$ and $I(t)=0$, respectively. Furthermore, we linearize the nonlinear stress-strain relation and reduce the \ac{TBT} to the \ac{EBBT}. This can be done by identifying the difference between the a prioir assumptions the Euler-Bernoulli beam theory and the Timoshenko beam theory are based upon. These assumptions are for the Euler-Bernoulli beam theory as follows,\\
            \begin{assumption}
            A priori Euler-Bernoulli beam assumptions \cite{carrera2011beam}
                \begin{enumerate}
                    \item \label{ch2:assump-A}
                        The cross-section is rigid on its plane.
                    \item \label{ch2:assump-B}%
                        The cross-section rotates around a neutral surface remaining plane.
                    \item \label{ch2:ass:aprioriEB_third}%
                    The cross-section remains perpendicular to the neutral surface during deformation.
                \end{enumerate}
            \end{assumption}
        Assumption 1.\ref{ch2:ass:aprioriEB_third} restricts the rotation $\phi(z)$ of an Euler-Bernoulli beam, which is only affected by pure bending, i.e. $\phi(z)\equiv w_z(z)$. For the Timoshenko beam theory, Assumption 1.\ref{ch2:ass:aprioriEB_third} is relaxed, and Assumptions 1.1 and 1.2 remain valid. Therefore, we have that the shear stresses are taken into account for \ac{TBT}, and allows for a different rotation of the beam. To reduce the \ac{TBT} to the \ac{EBBT}, we impose Assumption 1.\ref{ch2:ass:aprioriEB_third}, through $\phi\equiv w_z$ on the linearized version of \eqref{ch2:eq:PDE-nonlinVoss}. Hence, we recover the dynamics that satisfy the three a priori Euler-Bernoulli beam assumptions and obtain the linear and reduced piezoelectric actuator model as follows,
        \begin{subequations}
            \begin{align}
                \begin{split}\label{ch2:eq:linearizedVoss_FD_c2b_dyn}
                        \rho\left(A \ddot{v}-I_{0} \ddot{w}_{z_{1}}\right)&=C\left(A {v}_{z_1 z_{1}}-I_{0} w_{z_1 z_1 z_{1}}\right)-\gamma \beta A q_{z_1 z_{1}}\\   
                        \rho\left(I \ddot{w}_{z_{1}}-I_{0} \ddot{v }\right) &=C\left(I w_{z_1 z_{1} z_{1}}-I_{0} v_{z_{1} z_{1}}\right)+\gamma \beta I_{0} q_{z_{1} z_{1}} \\
                        \mu A \ddot{q} &=\beta A q_{z_1 z_{1}}-\gamma \beta\left(A v_{z_{1} z_{1}}-I_{0} w_{z_1 z_{1} z_{1}}\right)\\
                        \end{split}\\
                        \begin{split}\label{ch2:eq:linearizedVoss_FD_c2b_boun}
                        v(0)=w(0)=w_{z_1}(0)=0\\
                         A\beta q_{z_1}(0)-\gamma\beta(Av_{z_1}(0)-I_0 w_{z_1 z_1})(0)=0\\
                         {C}A v_{z_1}(1)-C{I_0} w_{{z_1}{z_1}}(1)-\gamma\beta A q_{z_1}(1)=0\\
                	    CI w_{{z_1}{z_1}}(1)-C{I_0} v_{{z_1}}(1)+\gamma\beta  I_0  q_{z_1}(1)=0\\
                	   \dot{q}(1)=0,
                \end{split}
            \end{align}
        \end{subequations}
        \blue{and total energy function similar to \eqref{ch2:eq:Hamiltonian_PDE_fd}. Furthermore, we have that the internal dynamics and energy function coincide with the well-posed piezoelectric actuator model \eqref{ch2:eq:PDE-piezoactuator} derived in this work. Therefore, we conclude that the models investigated here, i.e. \eqref{ch2:eq:PDE-piezoactuator-pde}, the piezoelectric beam in \cite{MenOSIAM2014} and the nonlinear piezoelectric actuator with \textit{current-through-the-boundary} actuation \eqref{ch2:eq:PDE-nonlinVoss} from \cite{VenSSIAM2014} are from the same 'family' of voltage-actuated piezoelectric systems. However, the boundary conditions and the rate of change of the energy differ, allowing for different behaviour and stabilizability results. From \cite{VenSSIAM2014}, we know that the spatially discretized non-linear piezoelectric actuator with Timoshenko beam theory satisfies Brocket's necessary conditions for stabilizability. However, for the infinite-dimensional model \eqref{ch2:eq:linearizedVoss_FD_c2b_dyn}, this is an open question.}\\
               
        Interestingly, the derivation of the voltage-controlled models in \cite{MenOSIAM2014} and \cite{VenSSIAM2014} have been derived using different modelling approaches. In \cite{MenOSIAM2014}, Hamilton's principle has been applied to a Lagrangian. Whereas in \cite{VenSSIAM2014}, the model of a transmission line and a nonlinear Timoshenko beam model have been interconnected in the port-Hamiltonian framework in a power-conserving way. Moreover, the current actuation of \eqref{ch2:eq:PDE-nonlinVoss} originating from \cite{VenSSIAM2014} is made possible by exploiting the boundary ports in the port-Hamiltonian framework and circumvents the use of a single integrator system \eqref{ch2:eq:singleintegrator} on the boundary required in the classical formalism, see Remark \ref{ch2:rem:currentthroughtheboundary}.
        


    \section{Derivation and comparison of piezoelectric composites}\label{ch2:sec:Piezoelectriccomposites}
   
        \begin{figure}%
            \centering
           \label{ch2:fig:Composite_options}{{\includegraphics[width=0.19\paperwidth]{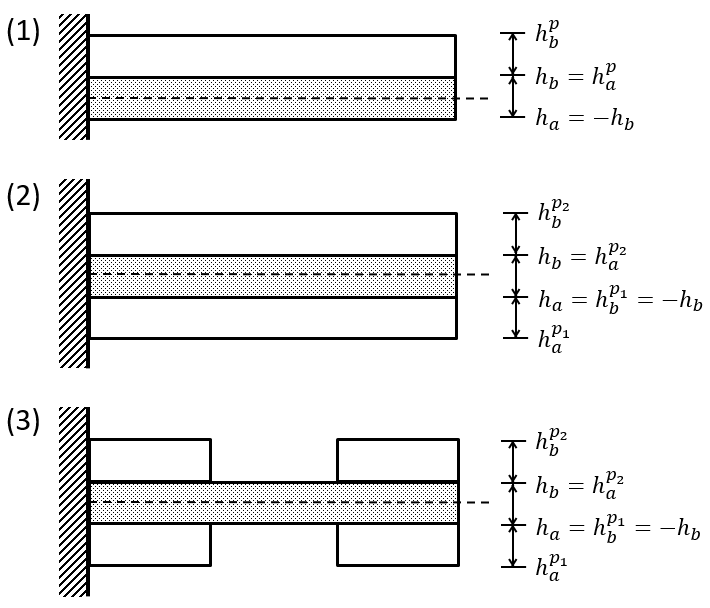} }}%
            \qquad
            \label{ch2:fig:bendingofabeam}{{ \includegraphics[width=0.15\paperwidth]{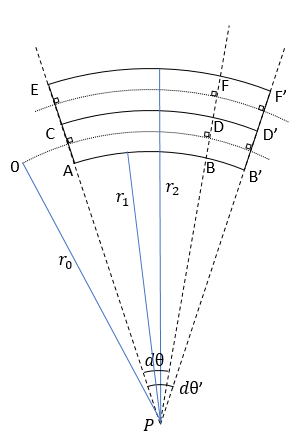} }}%
             \caption[Examples of piezoelectric composites and its bending]{
              (a) Examples of piezoelectric composite structures with the mechanical substrates and piezoelectric layers, respectively, in white and grey. The neutral axis is indicated by the dotted line. (1) Single-side piezoelectric layer allowing bending, (2) Sandwiched mechanical substrate for increased shear forces, depending on actuation, allows for bending/elongation. (3) Distributed piezoelectric layers are also possible for single-side actuation.
              (b) Bending of a beam around point $P$ with different angles $d\theta,d \theta'$.}%
        \end{figure}


        A piezoelectric composite is physically obtained by bonding at least one piezoelectric actuator onto a (purely) mechanical substrate and can be done in various ways; see Fig \ref{ch2:fig:Composite_options} for some of the flavours. Therefore, we can say that the piezoelectric actuator is a subsystem of the piezoelectric composite in a mathematical sense. In this section, we derive the piezoelectric composite model using linear Euler-Bernoulli beam theory with a piezoelectric layer on one side, see Fig \ref{ch2:fig:Composite_options}(1).\\ 
        Within the classical formalism, it is common to redo the application of Hamilton's principle to the Lagrangian associated with the piezoelectric composite, see for instance \cite{OzerMCSS2015,OzerTAC2019}, whereas, in the \ac{pH} approach, it is possible to interconnect different subsystems in a power-conserving way and obtain the mathematical model of the overall model directly. \blue{Then, it must also be possible to obtain an interconnected system model by interconnecting subsystems within this hybrid framework through careful considerations.} Here, we show this for interconnecting the piezoelectric actuator with a mechanical substrate to obtain the piezoelectric composite model. The piezoelectric actuator model \eqref{ch2:eq:PDE-piezoactuator} takes into account the rotational inertia's $I_0$, such that the deformation depending on the geometry of the interconnected piezoelectric actuator and mechanical substrate are treated properly. In contrast to the piezoelectric beam \eqref{ch2:eq:PDE-piezobeam}, where $I_0\downarrow0$ and requires redoing the derivation of the composite equations through Hamilton's principle. For the interconnection, a power-conserving interconnection is required. The power-conserving interconnection translates to \textit{equal speeds at the interconnection surface} ensuring the continuity of velocity and balanced exchange of force (equal yet opposed direction) \cite{port-HamiltonianIntroductory14}. For piezoelectric composites, the continuity of strain at the interconnection ensures the continuity of velocity. Moreover, the strain across the normal of the neutral axis remains the same under bending, which can  be seen from the computed strains 
        \begin{align*}
            \varepsilon_E&= \frac{\bar{EF'}-\bar{EF}}{\bar{EF}}=\frac{r_2d\theta'-r_2d\theta}{r_2d\theta}=\frac{d\theta'}{d\theta}-1,\\
            \varepsilon_A&=\frac{\bar{AB'}-\bar{AB}}{\bar{AB}}=\frac{r_1d\theta'-r_1d\theta}{r_1d\theta}=\frac{d\theta'}{d\theta}-1,
        \end{align*}
        of the bending composite in Fig \ref{ch2:fig:bendingofabeam}, even though the length of the curves differ at different heights, e.g. $\bar{EF}>\bar{AB}$.\\
        The mechanical substrate is a mechanical beam model, which can be derived using beam theory separately via Hamilton's principle. However, it can also be recovered from the piezoelectric actuator model by neglecting the piezoelectric effects \cite{MenOSIAM2014}. More precisely, in \eqref{ch2:eq:Constitutive_relations}, we see that the coupling between the mechanical domain $\braces{\sigma,\epsilon}$ and the electromagnetic domain $\braces{D,E}$ is evident through the piezoelectric constants matrix. By neglecting the effects of the coupling, Hooke's Law and the constitutive relations for a dielectric medium are recovered. For the piezoelectric actuator, the mechanic-electromagnetic coupling is a result of the piezoelectric constant $\gamma$, see \eqref{ch2:eq:preliminaries_piezoelectricConstitutiveRelaitons_scalar}. Therefore, by letting $\gamma\rightarrow 0$, the mechanical and electromagnetic domain become decoupled, resulting in a purely mechanical substrate model
        \begin{subequations}\label{ch2:eq:purelymechanicalsubstrate}
        \begin{align}
            \begin{split}
                 \rho\left(A \ddot{v}-I_{0} \ddot{w}_{z_{1}}\right)&=C\left(A {v}_{z_1 z_{1}}-I_{0} w_{z_1 z_1 z_{1}}\right)\\   
            \rho\left(I \ddot{w}_{z_{1}}-I_{0} \ddot{v }\right) &=C\left(I w_{z_1 z_{1} z_{1}}-I_{0} v_{z_{1} z_{1}}\right)\\
            \end{split}\\
        \begin{split}
        v(0)=w_{z_1}(0)=0\\
         {C}A v_{z_1}(1)-C{I_0} w_{{z_1}{z_1}}(1)\\
        	    CI w_{{z_1}{z_1}}(1)-C{I_0} v_{{z_1}}(1)=0.\\
        \end{split}
        \end{align}
        \end{subequations}
        The obtained mechanical substrate model \eqref{ch2:eq:purelymechanicalsubstrate} can be interconnected with the piezoelectric actuator \eqref{ch2:eq:PDE-piezoactuator} by assuming perfect bonding of the two layers. Then, we obtain the model for the piezoelectric composite on the spatial domain $z_1\in[0,1]$ as follows,
        \begin{subequations}\label{ch2:eq:PDE_composite_FD}
        \begin{align}
                	\begin{split}
                    	\rho_A\ddot{v}-\rho_{I_0}\ddot{w}_{z_1}&=C_A v_{z_1z_1}-C_{I_0}w_{z_1z_1z_1}-\gamma\beta A_p q_{z_1 z_1}\\
                    	\rho_I \ddot{w}_{z_1}-\rho_{I_0}\ddot{v}&=C_Iw_{z_1z_1z_1}-C_{I_0}v_{z_1z_1}+\gamma\beta I_0 q_{z_1 z_1}\\
                    	\mu A_p\ddot{q}&=\beta A_p q_{z_1z_1}-\gamma \beta\left(A v_{z_{1} z_{1}}-I_{0} w_{z_1 z_{1} z_{1}}\right),\\
                	\end{split}\\
         \begin{split}\label{ch2:eq:PDE_composite_bc}
                 v(0)=w(0)=w_z(0)=q(0)=0;\\
                 C_A v_z(\ell)-C_{I_0}w_{zz}(\ell)-\gamma\beta A q_{z_1}(\ell)=0,\\
                 C_Iw_{zz}(\ell)-C_{I_0}v_{z}(\ell)+\gamma\beta I_0 q_{z_1}(\ell)=0,\\ 
                 A\beta q_{z_1}(\ell)-\gamma\beta(Av_{z_1}(\ell)-I_0 w_{z_1 z_1}(\ell))=-AV(t),
         \end{split}
        \end{align}
        \end{subequations}
        with the concatenated system parameters
    	\begin{eqnarray}\label{ch2:eq:new_energycoefficients}
        		&\rho_A:=\rho_pA_p+\rho_s A_s,\qquad  &C_A:=C_pA_p+C_sA_s, \nonumber\\
        		&\rho_I:=\rho_pI_p+\rho_s I_s, \qquad &C_I:=C_pI_p+C_sI_s, \nonumber\\
        		&\rho_{I_0}:=\rho_p I_0, \qquad &C_{I_0}:=C_{11,p}I_0.
    	\end{eqnarray} 
    	where the subscripts $p,s$ indicate the system parameters associated with the piezoelectric actuator \eqref{ch2:eq:PDE-piezoactuator} and mechanical substrate \eqref{ch2:eq:purelymechanicalsubstrate}, respectively.\\
    	The total energy corresponding to the composite model \eqref{ch2:eq:PDE_composite_FD} is as follows
    	\begin{align}\label{ch2:eq:Hamiltonian_PDE_fd_comp}
        	\begin{split}
            	\mathcal{H}(t)&=\frac{1}{2}\int_0^\ell\left[\rho_A \dot{v}^2+\rho_I\dot{w}^2_{z_1}-2\rho_{I_0}\dot{v}\dot{w}_{z_1}+\mu A_p\dot{q}^2\right.  \\ &\quad\left.+C_Av_{z_1}^2+C_Iw_{z_1z_1}^2-2C_{I_0}v_{z_1}w_{z_1z_1}+\beta A_p q_{z_1}^2\right.\\
            	&\quad\left.-2 \gamma \beta\left(A v_{z_{1}}-I_{0} w_{z_{1}z_{1}}\right)q_{z_1}\right]dz_1.
        	\end{split}
    	\end{align} 
         This concludes the model derivation of the voltage-actuated piezoelectric composite. Interestingly, the structure and boundary conditions of the piezoelectric composite \eqref{ch2:eq:PDE_composite_FD} and piezoelectric actuator \eqref{ch2:eq:PDE-piezoactuator} are similar, although have different system parameters. Therefore, by following the same procedure as for the piezoelectric actuator \eqref{ch2:eq:PDE-piezoactuator}, it follows that the piezoelectric composite \eqref{ch2:eq:PDE_composite_FD} is well-posed. \\
         
         The piezoelectric composites with a static electric field and quasi-static electric field can be obtained from \eqref{ch2:eq:PDE_composite_FD} by the procedures discussed in Section \ref{ch2:sec:electromagneticconsiderations} and result in almost identical models with a difference in the stiffness of the attached piezoelectric composite, similarly to the piezoelectric beams \eqref{ch2:eq:(Q)staticbeams}. More precisely,
         \begin{subequations}\label{ch2:eq:(Q-)Staticcomps}     \begin{align}
            \begin{split}
                \rho_A \ddot{v}-\rho_{I_0}\ddot{w}_z & = \bar{C}_A v_{zz}-\bar{C}_{I_0}w_{zzz}\\
                 \rho_I\ddot{w}_z - \rho_{I_0}\ddot{v}&=\bar{C}_Iw_{zzz}-\bar{C}_{I_0}v_{zz}
                \end{split}\\
                \begin{split}
                    v(0)=&w_z(0) = 0\\
                 C_I w_{zz}(1)-&C_{I_0}v_z(1)=0\\
                 \bar{C}_Av_z(1)-&\bar{C}_{I_0} I_0^p)w_{zz}(1)=-\gamma A V(t) 
            \end{split}
      \end{align}
           \end{subequations}
         with the coefficients $\bar{C}_A$, $\bar{C}_I$, and $\bar{C}_{I_0}$ given in Table \ref{ch2:table:S&QS_coeff}, and energy functional
         \begin{align}\label{ch2:eq:Hamiltonian(Q-)Static_composite}
            \mathcal{H}_{(Q)S}(t)&=\frac{1}{2}\int_0^\ell  \left[ \rho_A \dot{v}^2+\rho_I\dot{w}^2_{z_1}-2\rho_{I_0}\dot{v}\dot{w}_{z_1}  \right.\\
            &\left.+\bar{C}_Av_{z_1}^2+\bar{C}_Iw_{z_1z_1}^2-2\bar{C}_{I_0}v_{z_1}w_{z_1z_1}\right]~dz_1. \nonumber
         \end{align}
         
         The piezoelectric composite models under the static or quasi-static electric field assumption have direct strain injection through the applied voltage by the absence of dynamics for the electromagnetic domain.
    
         \begin{table}
           \centering
            \begin{tabular}{ c | c | c }
                    & Static & Quasi-static   \\ \hline 
                $\bar{C}_A$ & $C_A-\gamma^2\beta A^p$ & ${C}_A$ \\ \hline
                $\bar{C}_I$ & $C_I-\gamma^2\beta \frac{(I_0^p)^2 }{A^p}$ & ${C}_I$     \\  \hline
                $\bar{C}_{I_0}$ & $C_{I_0}-\gamma^2\beta I_0^p$ & ${C}_{I_0}$    
            \end{tabular}
            \caption{Coefficients for Static and Quasi-static piezoelectric composite, corresponding to \eqref{ch2:eq:(Q-)Staticcomps}}
            \label{ch2:table:S&QS_coeff}
        \end{table}

         In \cite{VenSSIAM2014}, a nonlinear piezoelectric composite is derived with the use of the Timoshenko beam theory. They show that the spatially discretized system satisfies Brocket's necessary conditions \cite{Brockett83asymptoticstability} for stabilizability. Furthermore, the nonlinear fully dynamic piezoelectric composite model in \cite{VenSSIAM2014} can be linearized, and the Timoshenko beam can be reduced to the Euler-Bernoulli beam theory, such that an equivalent model to \eqref{ch2:eq:PDE_composite_FD} is obtained. In \cite{VenSSIAM2014}, a quasi-static piezoelectric actuator/composite model is proposed with current actuation. This model does not rely on the \textit{current-through-the-boundary} method we discussed in Remark \ref{ch2:rem:currentthroughtheboundary} and is, in fact, a purely current controlled piezoelectric model with quasi-static electric field assumption. Thus differing significantly from the piezoelectric composite models proposed in \cite{OzerTAC2019}. In \cite{MenOACC2014}, two linear piezoelectric composite models with Euler-Bernoulli and Timoshenko beam theory are derived. The composites in \cite{MenOACC2014} assume spatially distributed piezoelectric patches across the length of the mechanical substrate. The voltage-controlled piezoelectric composite model derived here \eqref{ch2:eq:PDE_composite_FD} assumes a piezoelectric layer across the whole length of the composite, which is useful for applications where one side of the composite has a specific function such as reflecting and focusing electromagnetic waves, e.g. in inflatable space structures and deformable mirrors. In the next section, we show that the fully dynamic voltage-controlled piezoelectric composite \eqref{ch2:eq:PDE_composite_FD} is \blue{useful for applications} by investigating its stabilizability properties.

    
    \section{Feedback stabilization of piezoelectric composites}\label{ch2:sec:Stabz}
    
            To stabilize the fully dynamic voltage-controlled piezoelectric composite model \eqref{ch2:eq:PDE_composite_FD}, we investigate a Lyapunov-based control strategy and make use of the following theorem for infinite dimensional systems.
                \begin{theorem}[LaSalle's Invariance Principle \cite{Luo1999SSIDSA}]\label{ch2:thm:LaSalleinv}
                    Let $\mathcal{V}$ be a continuous Lyapunov function for the strongly continuous semigroup $T(t)$ on $X$ and let the largest invariant subset be denoted as
                        $$\mathcal{W}:=\{\vect{x} \mid \tfrac{d}{dt}\mathcal{V}(\vect{x})=0\}. $$
                    If $x\in X$ and the orbit 
                        $$\tilde{\lambda}(\vect{x}) = \bigcup_{t \geq 0} T(t)\vect{x} $$
                     is precompact, then for the distance $d(.,.)$, we have that
                        $$\lim_{t\rightarrow\infty} d(T(t)\vect{x},\mathcal{W})=0.$$
                    Here, by invariance of $\mathcal{W}$ under $T(t)$, we mean $T(t)\mathcal{W}=\mathcal{W}$ for all $t\geq 0$. 
                    \qedwhite
                \end{theorem}
            Recall the corresponding energy functional $\mathcal{H}(t)$ in \eqref{ch2:eq:Hamiltonian_PDE_fd_comp}, with the change of energy along the trajectories of \eqref{ch2:eq:PDE_composite_FD} as follows,
            \begin{align}\label{ch2:eq:Lyap_roc}
                \frac{d}{dt}\mathcal{H}(t)=-V(t)\dot{q}(1)=-\kappa (\dot{q}(1))^2 \leq 0
            \end{align}
            for control choice $V(t)=\tfrac{\kappa}{A^p} \dot{q}(1)$ with $\kappa > 0$. Then, $\mathcal{H}(t)$ is considered a Lyapunov candidate function to prove the asymptotic stability of the closed system. Let $\vect{x}=\col(v_z,w_{zz},q_z,\rho_A \dot{v}-\rho_{I_0}\dot{w}_z,\rho_I \dot{w}_z - \rho_{I_0}\dot{v}, {\mu A^p } \dot{q})$ and denote the closed-loop system as follows,
            \begin{align}\label{ch2:eq:closedloop_sys}
                \begin{split}
                    \dot{\vect{x}}&=\mathcal{A}\vect{x}, t>0\\
                    \vect{x}(0)&=\vect{x}_0 \in \Dom(\mathcal{A})
                \end{split}
            \end{align}
            with the closed and densely defined operator
            \begin{align}\label{ch2:eq:closedloopgen}
                    \mathcal{A}\vect{x}&=\begin{pmatrix} \tfrac{\rho_I}{\rho_A\rho_I-\rho_{I_0}^2}x_4' + \tfrac{\rho_{I_0}}{\rho_A\rho_I-\rho_{I_0}^2} x_5'\\ \tfrac{\rho_A}{\rho_A\rho_I-\rho_{I_0}^2} x_5' + \tfrac{\rho_{I_0}}{\rho_A\rho_I-\rho_{I_0}^2}x_4'    \\  \tfrac{1}{\mu A^p} x_6'  \\
                    C_A x_1'-C_{I_0}x_2'-\gamma\beta A_p x_3'\\
                    C_Ix_2'-C_{I_0}x_1'+\gamma\beta I_0 x_3'\\
                    \beta A_p x_3'-\gamma \beta\left(A x_1' -I_{0} x_2'\right)  \end{pmatrix}
            \end{align}
            on the domain 
            \begin{align}
                        &\Dom(\mathcal{A})=\{ \vect{x} \in \mathcal{X}_\mathcal{H} |  \nonumber \\
                        & \quad C_A x_1(1)-C_{I_0}x_2(1)-\gamma\beta A x_3(1)=0,\\
                        & \quad C_Ix_2(1)-C_{I_0}x_1(1)+\gamma\beta I_0 x_3(1)=0, \nonumber\\ 
                       & \quad A\beta x_3(1)-\gamma\beta(A x_1(1)-I_0 x_2(1))+ \kappa {x_6}(1)=0 \}. \nonumber
            \end{align}
           It is straightforward to show that the operator \eqref{ch2:eq:closedloopgen} generates a strongly continuous semigroup of contractions $T(t)$, which is well-posed in a similar fashion as is shown for \eqref{ch2:eq:opA_FD}. Furthermore, we are able to prove asymptotic stability of the closed-loop system. Therefore, define the inner-product $\langle \vect{x},\vect{x}\rangle_\mathcal{H}=2\mathcal{H}(t)$ and assume the following inequalities on the system parameters.
           
           \begin{assumption}{System parameter inequalities for \eqref{ch2:eq:closedloop_sys}.}\label{ch2:ass:sys_coef}
                \begin{align*}
                    \begin{array}{rlrl}
                       \rho_A \rho_I &\neq \rho_{I_0}^2, \quad\quad   & C_I      &\neq    \gamma^2\beta \frac{(I_0^p)^2}{A^p}, \\
                        C_A C_I &\neq C_{I_0}^2,         \quad\quad   & C_{I_0} &\neq \gamma^2\beta A^p,    \\
                        C_A &\neq \gamma^2\beta A^p,   \quad\quad    & C_{I_0} &\neq n \gamma^2\beta  I_0^p, ~,
                    \end{array}
                \end{align*}
                {with $n\in\setof{1,2}$ }
           \end{assumption}
          
               \begin{theorem}\label{ch2:Thm:AStabz_FD_comp}
                    Let the inequalities in Assumption \ref{ch2:ass:sys_coef} hold and consider the closed-loop system  \eqref{ch2:eq:closedloopgen}. Furthermore, consider the Lyapunov candidate functions $\mathcal{V}=\mathcal{H}(t)$. Then, the closed-loop system is well-posed and asymptotically stable.
                \end{theorem}
                \begin{proof}
                    The closed and densely defined operator $\mathcal{A}$ and its adjoint $\mathcal{A}^\ast$ are dissipative, i.e.
                      \begin{align*}
                            \angles{\mathcal{A} \vect{x},\vect{x}}_X&=-\kappa (x_6(1))^2\leq 0, \\
                            \angles{\mathcal{A}^*{\vect{x}},\vect{x}}_X&=-\kappa (x_6(1))^2\leq 0,
                      \end{align*} 
                      where we computed $\mathcal{A}^\ast$ in a similar fashion as in Lemma \ref{ch2:lem:FD_adjoint}. Hence, by use of the Lummer-Phillips Theorem \ref{ch2:theorem:Lumer-Phillips}, the operator $\mathcal{A}$ generates a strongly continuous semigroup of contractions and is well-posed.\\
                    From the Sobolev embedding Theorem \cite{Sobolevembedding2006a} we have that $\Dom(\mathcal{A})$ is compact in $\mathcal{X}_\mathcal{H}$ and thus $\mathcal{A}$ is closed. Therefore, the resolvent of $\mathcal{A}$ is compact for all $\lambda$ in the resolvent set \cite{Kato1995}. Using   \cite{Luo1999SSIDSA}(Theorem 3.65), we see that the orbit  $\tilde{\gamma}(\vect{x})$ is precompact and the limit set is non-empty. It remains to show that the largest invariant subset 
                      $$\mathcal{W}=\{x \mid \dot{\mathcal{V}}(\vect{x})=0\},$$
                      with $x_6(1)\equiv 0$ contains only the zero vector $\vect{0}$. Therefore, let Assumption \ref{ch2:ass:sys_coef} hold and recall the boundary conditions $v(0)=w(0)=w_z(0)=q(0)=0$ and compute that the solution to the ode
                        \begin{align}
                            \mathcal{A}\vect{\varphi}(z) = \vect{0}, 
                        \end{align}
                      is $\vect{\varphi}(z)\equiv\vect{0}$ for $\vect{\varphi}\in\Dom(\mathcal{A})$. Hence, $\vect{0}\in \mathcal{W}$ is the only solution contained in $\mathcal{W}$. Then, by use of LaSalle's Invariance Principle Theorem \ref{ch2:thm:LaSalleinv}, we have that 
                      \begin{align}
                          \lim_{t\rightarrow\infty} d(T(t)\vect{x},\mathcal{W})=d(T(t)\vect{x},\vect{0})=0,
                      \end{align}
                      for all $\vect{x}\in \Dom(\mathcal{X})$ and conclude that the closed-loop system of the voltage-controlled fully dynamic electromagnetic piezoelectric composite is asymptotically stable. 
                  \end{proof}

                Furthermore, we show the asymptotic stability of the current-through-the-boundary-controlled fully dynamic electromagnetic piezoelectric composite \eqref{ch2:eq:linearizedVoss_FD_c2b_dyn}, which is of the same family as the asymptotically stabilizable voltage-controlled piezoelectric composite \eqref{ch2:eq:PDE_composite_FD}. In \cite{VenSSIAM2014}, it has been shown that the spatially discretized current-through-the-boundary-controlled piezoelectric composite satisfies Brocket's necessary condition for stabilizability. Here, we continue to prove the infinite-dimensional system is stabilizable.
            
                \begin{theorem}\label{ch2:prop:stabz_FD_c2b_composite}
                    Let the inequality $C\neq\gamma^2\beta$ hold additionally to  Assumption \ref{ch2:ass:sys_coef} and consider the closed-loop
                    operator \eqref{ch2:eq:closedloopgen} with domain
                \begin{align}\label{ch2:eq:domain_current2b}
                    \begin{split}
                            &\Dom(\mathcal{A}_{\dot{q}})=\{ \vect{x} \in \mathcal{X}_\mathcal{H} |  \\
                            & \quad C_A x_1(1)-C_{I_0}x_2(1)-\gamma\beta A x_3(1)=0,\\
                            & \quad C_Ix_2(1)-C_{I_0}x_1(1)+\gamma\beta I_0 x_3(1)=0,\\ 
                           & \tfrac{1}{\mu A^p} x_6+ \kappa {x_3}(1)=0 \},
                   \end{split}
                \end{align}
                    obtained by closing the loop of \eqref{ch2:eq:linearizedVoss_FD_c2b_dyn} though the control $\dot{q}(1)=\mathcal{I}(t)=-(A^p \beta (1-\tfrac{\gamma^2\beta}{C_A}))^{-1}q_z(1)$. Then, the closed-loop system is well-posed and asymptotically stable.
                \end{theorem}
                \begin{proof}
                    The closed and densely defined operator $\mathcal{A}_{\dot{q}}$ and its adjoint $\mathcal{A}_{\dot{q}}^\ast$ are dissipative, i.e.
                      \begin{align*}
                            \angles{\mathcal{A}_{\dot{q}} \vect{x},\vect{x}}_X&=-\kappa (x_3(1))^2\leq 0, \\
                            \angles{\mathcal{A}_{\dot{q}}^*{\vect{x}},\vect{x}}_X&=-\kappa (x_3(1))^2\leq 0,
                      \end{align*} 
                      where we made use of the boundary condition from the second row of \eqref{ch2:eq:domain_current2b}, and computed the adjoint $\mathcal{A}_{\dot{q}}^\ast$ in a similar fashion as in Lemma \ref{ch2:lem:FD_adjoint}. Hence, by use of the Lummer-Phillips Theorem \ref{ch2:theorem:Lumer-Phillips}, the operator $\mathcal{A}_{\dot{q}}$ on the domain generates a strongly continuous semigroup of contractions and is well-posed.\\
                    Furthermore, from the Sobolev embedding Theorem \cite{Sobolevembedding2006a} and using \cite{Luo1999SSIDSA} (Theorem 3.65), we see that the orbit  $\tilde{\gamma}(\vect{x})$ is precompact.  It remains to show that the largest invariant subset 
                      $$\mathcal{W}=\{x \mid \dot{\mathcal{V}}(\vect{x})=0\},$$
                      with \blue{${q}_z(1)\equiv 0$} contains only the zero vector $\vect{0}$. Therefore, let Assumption \ref{ch2:ass:sys_coef} and the inequality $C\neq\gamma^2\beta$ hold and recall the boundary conditions $v(0)=w(0)=w_z(0)=q(0)=0$ to compute the solution to the ode
                        \begin{align}
                            \mathcal{A}\vect{\varphi}(z) = \vect{0}, 
                        \end{align}
                      is $\vect{\varphi}(z)\equiv\vect{0}$ for $\vect{\varphi}\in\Dom(\mathcal{A}_{\dot{q}})$. Hence, $\vect{0}\in \mathcal{W}$ is the only solution contained in $\mathcal{W}$. Then, by use of LaSalle's Invariance Principle Theorem \ref{ch2:thm:LaSalleinv}, we have that 
                      \begin{align}
                          \lim_{t\rightarrow\infty} d(T(t)\vect{x},\mathcal{W})=d(T(t)\vect{x},\vect{0})=0,
                      \end{align}
                      for all $\vect{x}\in \Dom(\mathcal{X}_{\dot{q}})$ and conclude that the closed-loop system of the current-through-the-boundary fully dynamic electromagnetic composite is asymptotically stable.
                \end{proof}

              
               \blue{ Furthermore, for completion, we show the asymptotic stability of the current controlled piezoelectric composite under static and quasi-static electric field assumption \eqref{ch2:eq:(Q-)Staticcomps}. Recall the corresponding energy functional $\mathcal{H}_{(Q)S}(t)$ \eqref{ch2:eq:Hamiltonian(Q-)Static_composite}, with the change of energy along the trajectories of \eqref{ch2:eq:(Q-)Staticcomps} as follows,}
                \begin{align}
                    \frac{d}{dt}\mathcal{H}_{(Q)S}(t)=-\gamma A^p V(t)\dot{v}(t)=-\kappa (\dot{v}(1))^2\leq 0,
                \end{align}
                for control choice $V(t)=\tfrac{\kappa}{\gamma A^p}\dot{v}(1)$. Then, $\mathcal{H}_{(Q)S}(t)$ is considered a Lyapunov candidate function. Let $\vect{x}=\col(v_z,w_{zz},\rho_A \dot{v}-\rho_{I_0}\dot{w}_z,\rho_{I}\dot{w}_z - \rho_{I_0}\dot{v})$ and denote the closed-loop systems as follows,
                \begin{align}
                    \begin{split}
                        \dot{\vect{x}}\mathcal{A}_{(Q)S}\vect{x}, \quad t> 0, \\
                        \vect{x}(0)=\vect{x}_0 \in\Dom(\mathcal{A}_{(Q)S}),
                    \end{split}
                \end{align}
                with the closed and densely defined operator
                \begin{align}\label{ch2:eq:QS_closedloop_comp}
                    \mathcal{A}_{(Q)S}\vect{x}=\begin{pmatrix} \tfrac{\rho_I}{\rho_A\rho_I-\rho_{I_0}^2}x_3' + \tfrac{\rho_{I_0}}{\rho_A\rho_I-\rho_{I_0}^2} x_4'\\ \tfrac{\rho_A}{\rho_A\rho_I-\rho_{I_0}^2} x_4' + \tfrac{\rho_{I_0}}{\rho_A\rho_I-\rho_{I_0}^2}x_3'    \\  
                    \bar{C}_A x_1'-\bar{C}_{I_0}x_2'\\
                    \bar{C}_Ix_2'-\bar{C}_{I_0}x_1' \end{pmatrix},
                \end{align}
                on the domain
                \begin{align}
                    \begin{split}
                        \Dom(\mathcal{A}_{(Q)S})=\{ \vect{x} \in \mathcal{X}_{(Q)S} &|  \\
                         \quad \bar{C}_A x_1(1)-\bar{C}_{I_0}x_2(1)&=0,\\
                         \quad \bar{C}_Ix_2(1)-\bar{C}_{I_0}x_1(1)+ \kappa \braces(1)&=0,\\ 
                        \quad A\beta x_3(1)-\gamma\beta(A x_1(1)-I_0 x_2(1))&\\
                       \quad+ \kappa \tfrac{\rho_I x_3(1)+\rho_A x_4(1)}{\rho_A\rho_I-\rho_{I_0}^2}&=0 \}. 
               \end{split}
                \end{align}
                Define the innerproduct $\langle\vect{x},\vect{x}\rangle_{(Q)S}=2\mathcal{H}_{(Q)S} (t)$ and assume the following inequalities on the system parameters.
           
           \begin{assumption}{System parameter inequalities for \eqref{ch2:eq:QS_closedloop_comp}}\label{ch2:ass:sys_coefQS}
                \begin{align*}
                    \begin{array}{rlrl}
                       \rho_A \rho_I &\neq \rho_{I_0}^2, \quad\quad   & \bar{C}_A \bar{C}_I &\neq \bar{C}_{I_0}^2. \\
                    \end{array}
                \end{align*}
           \end{assumption}

              \blue{Now we show the asymptotically stabilization of for both the static and quasi-static piezoelectric composite with the following Theorem. } 
                \begin{theorem}\label{ch2:Thm:AStabz_(Q)S_comp}
                   Let the inequalities in Assumption \ref{ch2:ass:sys_coefQS} hold and consider the closed-loop system \ref{ch2:eq:QS_closedloop_comp}. Furthermore, consider the Lyapunov candidate function $\mathcal{V} =  \mathcal{H}_{(Q)S}$.  Then, the closed-loop  (quasi-)static voltage-controlled piezoelectric composite is well-posed and asymptotically stable.
                \end{theorem}
                \begin{proof}
                     The closed and densely defined operator $\mathcal{A}_{(Q)S}$ and its adjoint $\mathcal{A}_{(Q)S}^\ast$ are dissipative, i.e.
                      \begin{align*}
                            \angles{\mathcal{A}_{(Q)S} \vect{x},\vect{x}}_X&=-\kappa (\dot{v}(1))^2\leq 0, \\
                            \angles{\mathcal{A}_{(Q)S}^*{\vect{x}},\vect{x}}_X&=-\kappa (\dot{v}(1))^2\leq 0,
                      \end{align*} 
                      where we computed $\mathcal{A}_{(Q)S}^\ast$ in a similar fashion as in Lemma \ref{ch2:lem:FD_adjoint}. Hence, by use of the Lummer-Phillips Theorem \ref{ch2:theorem:Lumer-Phillips}, the operator $\mathcal{A}_{(Q)S}$ generates a strongly continuous semigroup of contractions and is well-posed.\\
                    Furthermore, from the Sobolev embedding Theorem \cite{Sobolevembedding2006a} and using \cite{Luo1999SSIDSA}(Theorem 3.65 ), we see that the orbit  $\tilde{\gamma}(\vect{x})$ is precompact.  It remains to show that the largest invariant subset 
                      $$\mathcal{W}=\{x \mid \dot{\mathcal{V}}(\vect{x})=0\},$$
                      with $\dot{v}(1)\equiv 0$ contains only the zero vector $\vect{0}$. Therefore, let Assumption \ref{ch2:ass:sys_coefQS} hold and recall the boundary conditions $v(0)=w(0)=w_z(0)=0$ and compute that the solution to the ode
                        \begin{align}
                            \mathcal{A}_{(Q)S}\vect{\varphi}(z) = \vect{0}, 
                        \end{align}
                      is $\vect{\varphi}(z)\equiv\vect{0}$ for $\vect{\varphi}\in\Dom(\mathcal{A}_{(Q)S})$. Hence, $\vect{0}\in \mathcal{W}$ is the only solution contained in $\mathcal{W}$. Then, by use of LaSalle's Invariance Principle Theorem \ref{ch2:thm:LaSalleinv}, we have that 
                      \begin{align}
                          \lim_{t\rightarrow\infty} d(T(t)\vect{x},\mathcal{W})=d(T(t)\vect{x},\vect{0})=0,
                      \end{align}
                      for all $\vect{x}\in \Dom(\mathcal{X}_{(Q)S})$ and conclude that the closed-loop system is asymptotically stable. 
                \end{proof}
                
                \blue{In this section, we have shown that the voltage-actuated piezoelectric composite is asymptotically stabilizable under any of the three electromagnetic assumptions. For simulation and control purposes, we include simulations of the voltage-controlled fully dynamic electromagnetic piezoelectric composite in the following section. }  

            \section{Simulation results of the asymptotically stabilizing piezoelectric composite}\label{ch2:sec:simulations}
            
                 For simulation purposes, we consider the fully dynamic electromagnetic piezoelectric composite with the top layer being a \ac{PZT}-5\footnote{\text{https://support.piezo.com/article/62-material-properties}} piezoelectric layer and the mechanical substrate (with centroidal coordinates) is a steel-304\footnote{https://support.piezo.com/article/62-material-properties\#pack} mechanical layer of the same dimensions. An overview of the system parameters is given in Table \ref{ch2:tab:symbols}. The open-loop and closed-loop simulations, by closing the loop in a standard passive manner \cite{Jacobzwart2012}, are obtained using the structure-preserving spatial discretization method \cite{GoloSchaft2004} with N=20 segments. The time-discretization is complimented with the variable-step ode23s-solver (build-in Matlab\textsuperscript{\textregistered} solver).  Furthermore, to overcome difficulties with the spatial discretization of (piezoelectric) models using a mixed Finite-Element method \cite{GoloSchaft2004}, mentioned in \cite{voss2010port}, we use a trapezoidal spatial integration and time integration to compute the longitudinal $(v(z,t))$ and transversal deflection $w(z,t)$, respectively. \\
                In Fig  \ref{ch2:fig:sim_openloop}, the open-loop dynamics for the transverse tip deflection is presented after the actuation with 500V for two seconds. The asymptotic behaviour of the transverse tip deflection of the closed-loop system $\eqref{ch2:eq:closedloop_sys}$ with $\kappa=10$ is shown in Fig \ref{ch2:eq:closedloop_sys}, where it can be seen that the transversal tip deflection is mitigated. Furthermore, in Fig \ref{ch2:fig:controlled_statetraj}, the state-trajectories of the controlled piezoelectric composite are presented, which go to zero.

             \begin{figure}[ht]
         				\centering
            	    	\includegraphics[width=\columnwidth]{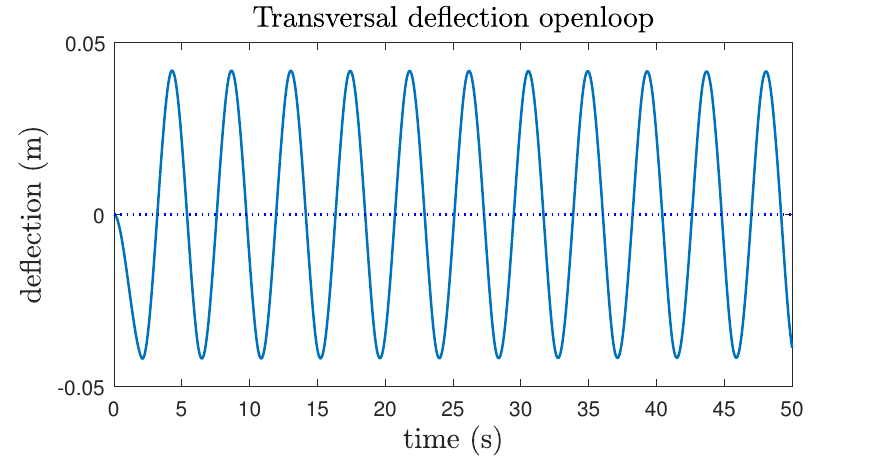}
            				\caption[Open-loop tip deflection of voltage-actuated piezoelectric composite]{Open-loop tip deflection $(w(1))$ of the piezoelectric composite. The piezoelectric composite is actuated with 500 V for the first two seconds and reaches a deflection of four cm at the tip. After two second, the actuation is stopped and it can be seen that the tip of the composite moves back to zero and continues to vibrate in a regular fashion.}
                \label{ch2:fig:sim_openloop}
            \end{figure}
              \begin{figure}[ht]
            				\centering
            				\includegraphics[width=\columnwidth]{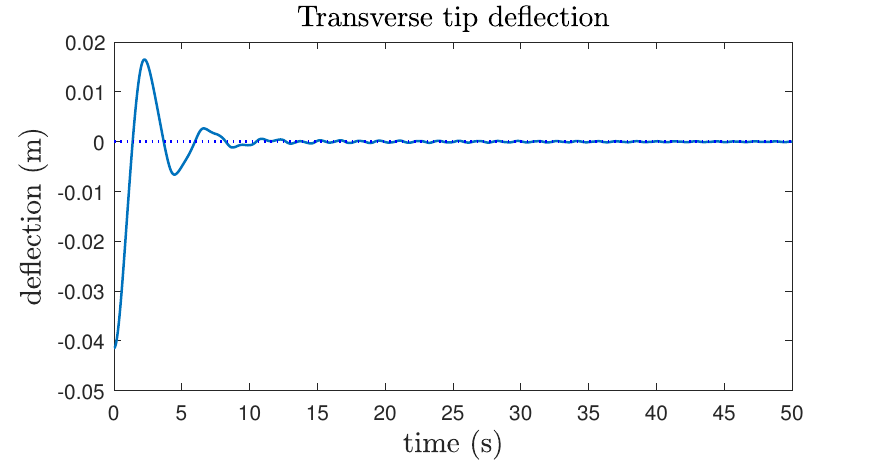}
            				\caption[Closed-loop tip deflection of voltage-controlled piezoelectric composite]{Closed-loop tip deflection $(w(z))$ - with $\kappa = 10$ after being actuated with 500 V for the first two seconds. Then the asymptotic behaviour of the controlled piezoelectric composite can be seen.}
            			\label{ch2:fig:sim_closedloop}
        	\end{figure}

             \begin{figure}%
                \centering
                    {{\includegraphics[width=0.99\columnwidth]{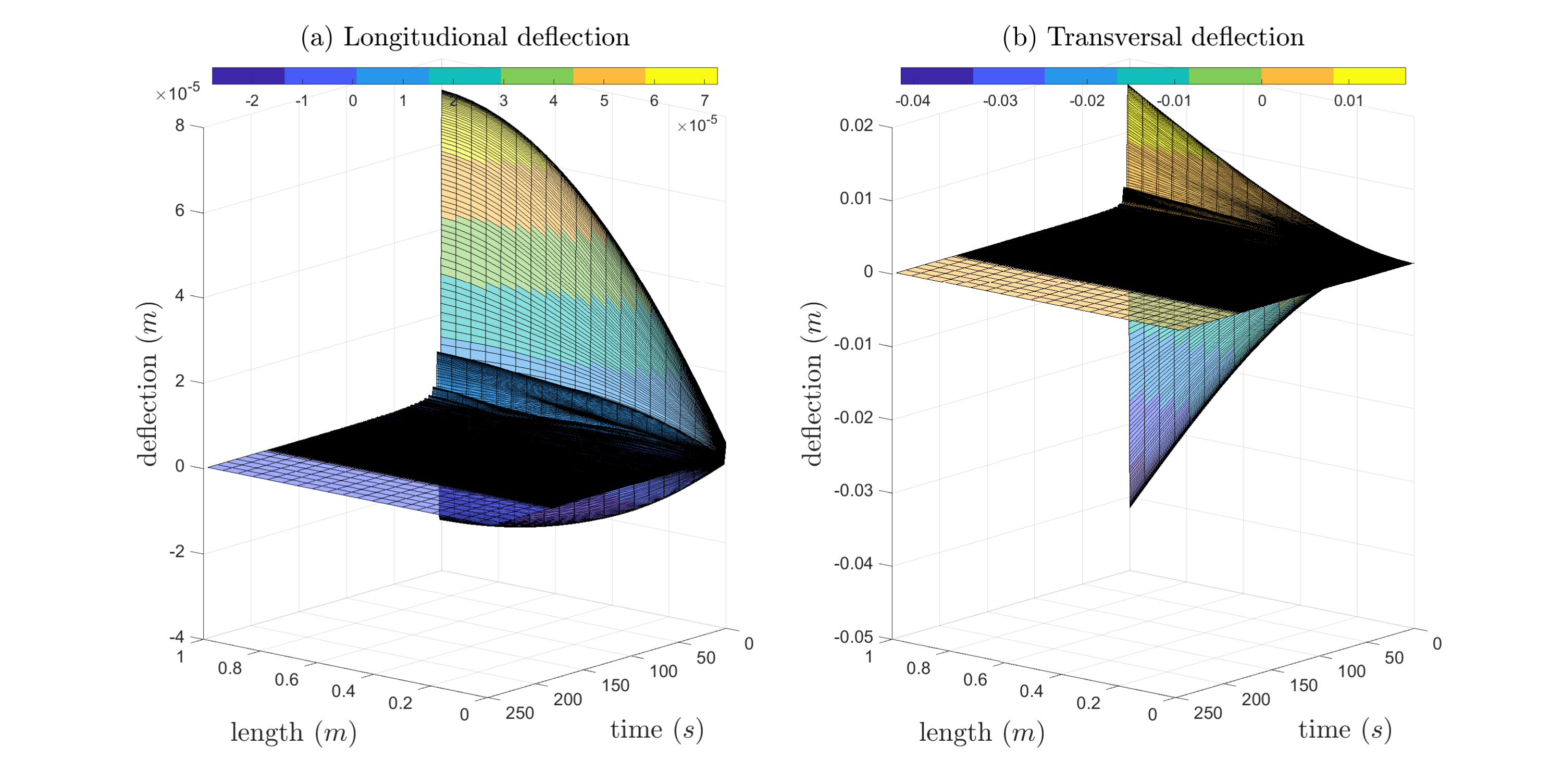} }}%
                    \quad
                    {{ \includegraphics[width=0.99\columnwidth]{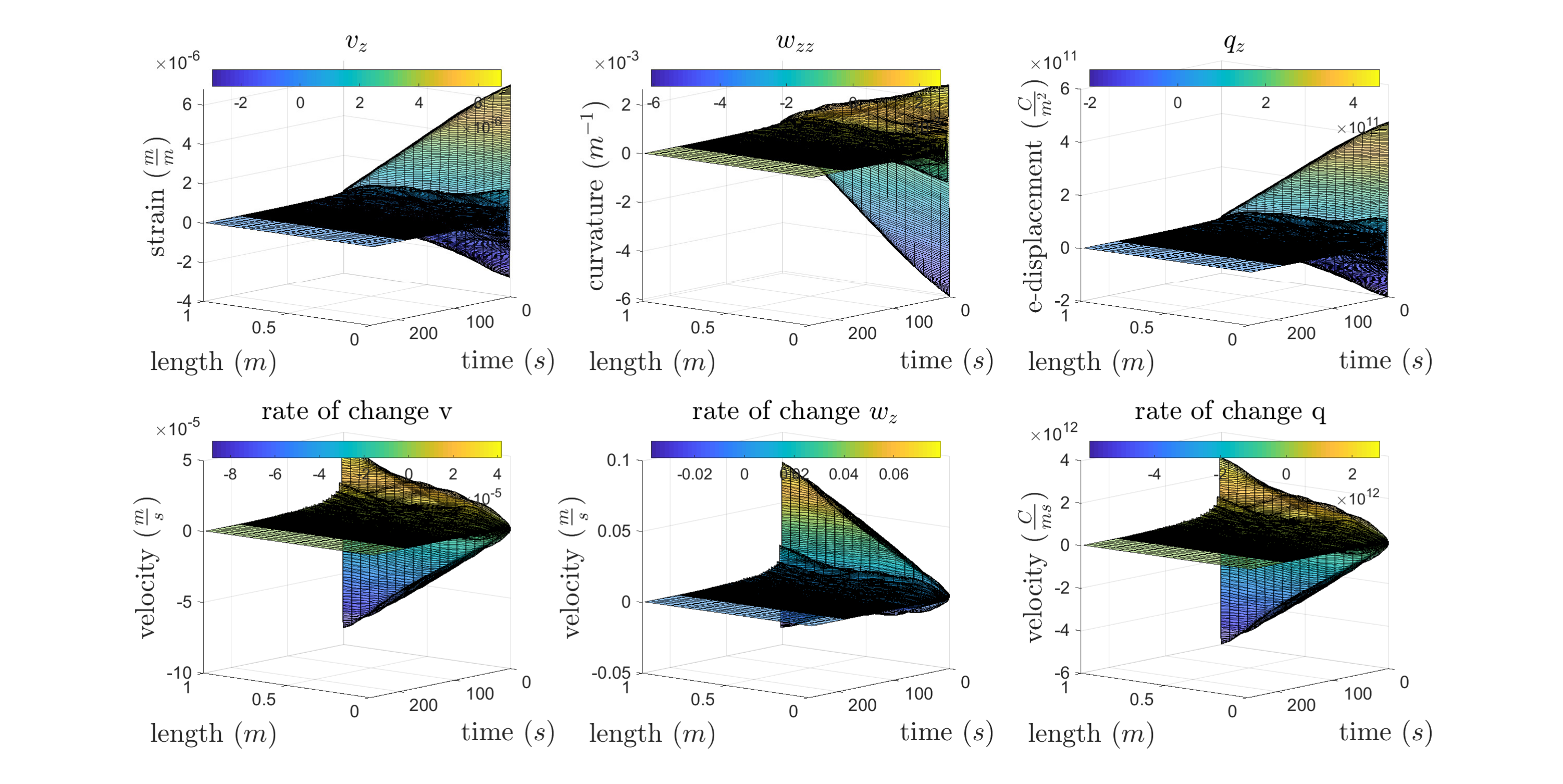} }}%
                 \caption[Closed-loop response form the electrical feedback with $\kappa=10$.]{In (a) the asymptotic stabilizing behaviour of the voltage-controlled longitudinal $(v(z))$ and transverse displacement $(w(z)$ of the complete composite is shown, for $\kappa=10$. In (b), the state trajectories (including velocities) of the voltage-controlled piezoelectric composite are shown to go to zero, for $\kappa=10$.}
        		  \label{ch2:fig:controlled_statetraj}
        \end{figure}

               
            %
             \begin{table}
                    \centering
                    \caption{System parameters}
                    {\begin{tabular}{lll}
                        \toprule
                            Geometry & Description& Value\\
                        \midrule
                            $2g_b$  & Layer width & $0,1 ~m$ \\
                            $h_b-h_a$ & Layer thicness & $0,01 ~m$\\
                        \midrule
                             Piezo parameters     \\
                        \midrule
                            $\rho_p$    &  mass-density     & $7950 ~\frac{\text{kg}}{\text{m}^3}$\\
                            $C_p$         & Stiffness         & $66\times10^9 ~\frac{\text{N}}{\text{m}^2}$\\
                            $\gamma$    &   Coupling coefficient &  $12.54  ~{\frac{\text{C}}{\text{m}^2}}$\\
                            $\beta$     & Impermittivity & $10^{-6}~ \frac{m}{F}$ \\
                            $\mu$ & Magnetic permeability & $1.2 \times 10^{-6} ~\frac{H}{m}$  \\
                            \midrule
                            Substrate parameters &  &  \\
                            \midrule
                            $\rho_s$    &  mass-density     & $8000v ~\frac{\text{kg}}{\text{m}^3}$\\
                            $C_s$       & Stiffness         & $193\times10^9  ~\frac{\text{N}}{\text{m}^2}$\\
                        \bottomrule
                        \end{tabular}}
                    \label{ch2:tab:symbols}
                \end{table}


    \section{Conclusion and future work}\label{ch2:sec:DF}
       
       \blue{ In this work, we propose a new well-posed voltage-actuated piezoelectric actuator and composite model with a fully dynamic electromagnetic field and show that the piezoelectric composite is asymptotically stabilizable for certain system parameters. Furthermore, we show that for certain system parameters, the composites under static and quasi-static electric field assumptions are also asymptotically stabilizable. Therefore, we claim that the voltage-controlled piezoelectric composites under any electromagnetic assumption are stabilizable for certain system parameters. Furthermore, we show that the non-linear \textit{current-through-the-boundary}-controlled piezoelectric composite proposed in \cite{VenSSIAM2014} that exploits the boundary ports in the \ac{pH}-formalism is of the same family of the voltage-controlled piezoelectric systems. The spatial discretization of this particular model satisfies Brocket's necessary condition for stabilizability \cite{VenSSIAM2014}, and we improved upon this by showing the asymptotic stabilizability for the infinite-dimensional model in the linear case for certain system parameters. Furthermore, we show that the piezoelectric beam is a special case of the piezoelectric actuator by neglecting the rotational inertia (by means of centroidal coordinates). We show that the piezoelectric actuator is useful for developing piezoelectric composite models, where a piezoelectric actuator is bonded to a mechanical substrate, as opposed to the piezoelectric beam, which requires a redo of the model derivation as the rotational inertias are neglected in the strain expressions. Difficulties with the spatial discretization of (piezoelectric) models using a mixed Finite-Element method \cite{GoloSchaft2004}, mentioned in \cite{voss2010port}, are solved by the use of a trapezoidal spatial integration and time integration to compute the longitudinal and transverse deflection. A full connection with the various current controlled models, e.g. \cite{OzerTAC2019} and \cite{VenSSIAM2014}, remains and is part of future research.}
        
        

    \section*{Acknowlegdement}
        {We would like to acknowledge Kirsten A. Morris for her insights and suggestions during the development of this work.}

    \bibliographystyle{IEEEtran}
	
	
\end{document}